\documentclass[10pt]{article}
\usepackage{amsmath} %
\usepackage[T1]{fontenc}
\usepackage[dvips]{color}
\usepackage{subfigure}
\usepackage{graphicx}
\definecolor{dgreen}{rgb}{0,.8,.3}
\definecolor{lblue}{rgb}{.2,.3,.7}

\newtheorem{assumption}{Assumption}
\newtheorem{theorem}{Theorem}
\newtheorem{definition}{Definition}
\newtheorem{lemma}{Lemma}
\newtheorem{proposition}{Proposition}
\newtheorem{corollary}{Corollary}

\newtheorem{remark}{Remark}

\numberwithin{equation}{section}
\numberwithin{lemma}{section}
\numberwithin{theorem}{section}

\newcommand{\beq}{\begin{equation}}
\newcommand{\eeq}{\end{equation}}
\newcommand{\beqa}{\begin{eqnarray}}
\newcommand{\eeqa}{\end{eqnarray}}
\newcommand{\beqas}{\begin{eqnarray*}}
\newcommand{\eeqas}{\end{eqnarray*}}
\newcommand{\ba}{\begin{array}}
\newcommand{\ea}{\end{array}}
\newcommand{\bi}{\begin{itemize}}
\newcommand{\ei}{\end{itemize}}
\newcommand{\gap}{\hspace*{1em}}

\newcommand{\nn}{\nonumber}

\def\Arg{{\rm Arg}}
\def\be{{\bf e}}

\def\bt{{\bar t}}
\def\bx{{\bar x}}
\def\c{{\rm c}}
\def\cB{{\cal B}}
\def\cC{{\cal C}}
\def\cFr{{\cC_s \cap \Omega}}

\def\cK{{\cal K}}
\def\cN{{\cal N}}
\def\cO{{\cal O}}
\def\cp{{\cal P}}
\def\cQ{{\cal Q}}
\def\cS{{\cal S}}
\def\cT{{\cal T}}

\def\cX{{\cal X}}
\def\cs{{\it ChangeSupport}}
\def\cswap{{\it SwapCoordinate}}

\def\hcK{{\hat \cK}}
\def\hatt{{\hat t}}
\def\hx{{\hat x}}
\def\Mod{{\rm mod}}
\def\proj{{\rm Proj}}
\def\q{{q}}

\def\sign{{\rm sign}}
\def\supp{{\rm supp}}
\def\tcK{{\tilde \cK}}
\def\tomega{{\widetilde\Omega}}
\def\tsigma{{\tilde \Sigma}}

\def\tx{{\tilde x}}
\def\T{{\textsc{\bf T}}}

\title{Optimization over Sparse Symmetric Sets via a Nonmonotone Projected Gradient Method}

\author{
 Zhaosong Lu
 \thanks{Department of Mathematics, Simon Fraser University, Canada 
(Email: {\tt zhaosong@sfu.ca}). This author was supported in part by NSERC Discovery Grant.}
}

\date{November 21, 2015}

\begin{document}

\maketitle

\begin{abstract}

We consider the problem of minimizing a Lipschitz differentiable function over a class of sparse 
symmetric sets that has wide applications in engineering and science. For this problem, it is known that 
any accumulation point of  the classical projected gradient (PG) method with a constant stepsize $1/L$ satisfies the $L$-stationarity optimality condition that was introduced in \cite{BeHa14}.  In this paper we introduce a new optimality condition that is stronger than the $L$-stationarity optimality condition. We also propose a nonmonotone projected gradient (NPG) method for this problem by  incorporating some support-changing and coordinate-swapping strategies into a projected gradient method with variable stepsizes. It is shown that any accumulation point of NPG satisfies the new optimality condition and moreover it is a coordinatewise 
stationary point. Under some suitable assumptions, we further show that it is a {\it global} or a {\it local} minimizer of 
the problem. Numerical experiments are conducted to compare the performance of PG and NPG. The 
computational results demonstrate that NPG has substantially better solution quality than PG, and moreover,  
it is at least comparable to, but sometimes can be much faster than PG in terms of speed.

\end{abstract}

\section{Introduction}

Over the last decade sparse solutions have been concerned in numerous applications. For example, in 
compressed sensing, a large sparse signal is decoded by using a sampling matrix and a relatively 
low-dimensional measurement vector, which is typically formulated as minimizing a least squares 
function subject to a cardinality constraint (see, for example, the comprehensive reviews 
\cite{TrWr10,DaDuElKu11}).  As another example, in financial industry portfolio managers 
often face business-driven requirements that limit the number of constituents in their tracking 
portfolio. A natural model for this is to minimize a risk-adjusted return over a 
cardinality-constrained simplex, which  has recently been considered in \cite{TaNiGoKa12,KyBeCeKo13, XuLuXu15, BeHa14} 
for finding a sparse index tracking. These models can be viewed as a special case of the following 
general cardinality-constrained optimization problem:
\beq \label{sparse-prob}
f^* := \min\limits_{x\in\cFr} f(x),
\eeq
where $\Omega$ is a closed convex set in $\Re^n$ and 
\[
\cC_s = \{x\in\Re^n: \|x\|_0 \le s\}
\]
for some $s\in\{1,\ldots,n-1\}$. Here, $\|x\|_0$ denotes the cardinality or the number of nonzero elements of $x$, and 
$f:\Re^n \to \Re$ is Lipschitz continuously differentiable, that is,  there is a constant $L_f > 0$ such that
\beq \label{lipschitz}
\left\| {\nabla f(x) - \nabla f(y)} \right\| \le L_f\left\| {x - y} \right\| \quad \forall x, y \in \Re^n.
\eeq

Problem \eqref{sparse-prob} is generally NP-hard. One popular approach to finding an approximate 
solution of \eqref{sparse-prob} is by convex relaxation. For example, one can replace the associated 
$\|\cdot\|_0$  of \eqref{sparse-prob} by $\|\cdot\|_1$ and the resulting problem has a convex feasible region. Some 
well-known models in compressed sensing or sparse linear regression such as lasso \cite{Ti96}, basis pursuit 
\cite{ChDoSa98}, LP decoding \cite{CanTao05} 
 and the Dantzig selector \cite{CanTao07}  
were developed in this spirit. In addition, direct approaches have been proposed in the literature for solving some 
special cases of \eqref{sparse-prob}. For example, IHT \cite{BlDa08,BlDa09} and CoSaMP \cite{NeTr09} are two direct 
methods for solving problem \eqref{sparse-prob} with $f$ being a least squares function and $\Omega=\Re^n$. Besides, 
the IHT method was extended and analyzed in \cite{Lu14} for solving 
$\ell_0$-regularized convex cone programming. The gradient 
support pursuit (GraSP) method was proposed in \cite{BaRaBo13} for solving \eqref{sparse-prob} 
with a general $f$ and $\Omega=\Re^n$. A penalty decomposition method was introduced and 
studied in \cite{LuZh13} for solving \eqref{sparse-prob} with general $f$ and $\Omega$.

Recently, Beck and Hallak \cite{BeHa14} considered problem \eqref{sparse-prob} in which $\Omega$ 
is assumed to be a symmetric set.  They introduced three types of optimality conditions 
that are basic feasibility, $L$-stationarity and coordinatewise optimality, and established a hierarchy 
among them. They also proposed methods for generating points satisfying these optimality conditions. 
Their methods require finding an {\it exact} solution of a sequence of subproblems in the form of 
\[
\min \{f(x): x \in \Omega, \ x_i =0 \  \ \forall i \in I \}
\]
for some index set $I \subseteq \{1,\ldots,n\}$. Such a requirement is, however, 
generally hard to meet unless $f$ and $\Omega$ are both sufficiently simple. This 
motivates us to develop a method suitable for solving problem \eqref{sparse-prob} 
with more general $f$ and $\Omega$. 

As studied in \cite{BeHa14}, the orthogonal projection of a point onto $\cFr$ can be 
efficiently computed for some symmetric closed convex sets $\Omega$. It is thus suitable to apply the projected 
gradient (PG) method with a fixed stepsize $t \in(0,1/L_f)$ to solve \eqref{sparse-prob} with such $\Omega$. 
It is easily known that any accumulation point $x^*$ of the sequence generated by PG satisfies 
\beq \label{nonconvex-fixpt}
x^* \in \Arg\min\left\{\|x-(x^*-t\nabla f(x^*))\|: x\in \cFr\right\}.\footnote{By convention, the symbol $\Arg$ stands for the set of the solutions of the associated 
optimization problem. When this set is known to be a singleton, we use the symbol $\arg$ to stand for it instead.} 
\eeq
It is noteworthy that this type of convergence result is weaker than that of the PG method applied to 
the problem 
\[
\min \{f(x): x\in\cX\}, 
\]  
where $\cX$ is a closed convex set in $\Re^n$. For this problem, it is known that 
any accumulation point $x^*$ of the sequence generated by PG with a fixed stepsize 
$t \in(0,1/L_f)$ satisfies 
\beq \label{convex-fixpt}
x^* = \arg\min\left\{\|x-(x^*-t\nabla f(x^*))\|: x\in \cX\right\}. 
\eeq  
The uniqueness of solution to the optimization problem involved in \eqref{convex-fixpt} is due to the convexity 
of $\cX$. Given that $\cFr$ is generally nonconvex, the solution to the optimization problem involved in 
\eqref{nonconvex-fixpt} may not be unique. As shown in later section of this paper, if it has a distinct 
solution $\tx^*$,   that is,   
\[
x^* \neq \tx^* \in \Arg\min\left\{\|x-(x^*-t\nabla f(x^*))\|: x\in \cFr\right\},
\]
then $f(\tx^*)<f(x^*)$ and thus $x^*$ is certainly not an optimal solution of 
\eqref{sparse-prob}. Therefore, a convergence result such as 
\beq \label{nonconvex-fixpt1}
x^* = \arg\min\left\{\|x-(x^*-t\nabla f(x^*))\|: x\in \cFr\right\} 
\eeq
is generally stronger than \eqref{nonconvex-fixpt}. 

In this paper we first study some properties of the orthogonal projection of a point onto $\cFr$ and propose a 
new optimality condition for problem \eqref{sparse-prob}. We then propose a nonmonotone projected gradient 
(NPG) method \footnote{As mentioned in the literature (see, for example, \cite{FeLuRo96,GrLu86,
ZhHa04}), nonmonotone type of methods often produce solutions of better quality than the monotone 
counterparts for nonconvex optimization problems, which motivates us to propose a  
nonmonotone type method in this paper.} for solving problem \eqref{sparse-prob}, which incorporates some  
support-changing and coordinate-swapping strategies into a PG method with variable stepsizes. It is shown that 
any accumulation point $x^*$  of the sequence generated by NPG satisfies \eqref{nonconvex-fixpt1} for all 
$t\in [0,\T]$ for some $\T\in (0,1/L_f)$. Under some suitable assumptions, we further show 
that $x^*$ is a  coordinatewise stationary point. Furthermore, if $\|x^*\|_0<s$,  
then $x^*$ is a {\it global} optimal solution of \eqref{sparse-prob}, and it is a local optimal 
solution otherwise. We also conduct numerical experiments to compare the performance of NPG and the PG method with a fixed
 stepsize. The 
computational results demonstrate that NPG has substantially better solution quality than PG, and moreover,  
it is at least comparable to, but sometimes can be much faster than PG in terms of speed.  

The rest of the paper is organized as follows. In section \ref{proj-set}, we study some properties of the orthogonal 
projection of a point onto $\cFr$.  In section \ref{opt-condns}, we propose a new optimality condition for problem 
\eqref{sparse-prob}.  In section \ref{method} we propose an NPG method for solving problem \eqref{sparse-prob} 
and establish its convergence.  We conduct numerical experiments in section \ref{results} to compare the performance 
of the NPG and PG methods. Finally, we present some concluding remarks in section \ref{conclude}.

\subsection{Notation and terminology} \label{notation}


For a real number $a$, $a_+$ denotes the nonnegative part of $a$, that is, $a_+ = \max\{a,0\}$. The symbol $\Re^n_+$ denotes the nonnegative orthant of $\Re^n$. Given any $x\in \Re^n$, $\|x\|$ is the Euclidean 
norm of $x$ and $|x|$ denotes the absolute value of $x$, that is, $|x|_i = |x_i|$ for all $i$. 
In addition, $\|x\|_0$ denotes the number of nonzero entries of $x$. The 
 support set of $x$ is defined as $\supp(x) = \{i: x_i \neq 0\}$. Given an index set $T \subseteq \{1,\ldots, n\}$, $x_T$ denotes the sub-vector 
of $x$ indexed by $T$, $|T|$ denotes the cardinality of $T$, and $T^\c$ is the complement 
of $T$ in $\{1,\ldots, n\}$.  
For a set $\Omega$, we define
$\Omega_T =  \{x\in\Re^{|T|}: \sum_{i\in T} x_i \be_i \in \Omega\}$,
where $\be_i$ is the $i$th coordinate vector of $\Re^n$. 

Let $s\in\{1,\ldots,n-1\}$ be given. Given any $x\in\Re^n$ with $\|x\|_0 \le s$, the index set $T$ is called a 
$s$-super support of $x$ if $T \subseteq \{1,\ldots, n\}$ satisfies $\supp(x) \subseteq T$ and $|T| \le s$.  
The set of all s-super supports of $x$  is denoted by 
$\overline\cT_s(x)$, that is, 
\[
\overline\cT_s(x) = \left\{T \subseteq \{1,\ldots,n\}: \supp(x) \subseteq T \ \mbox{and} \ |T| \le s \right\}.
\]
In addition, $T\in\overline\cT_s(x)$ is called a $s$-super support of $x$ with cardinality $s$ if $|T|=s$. 
The set of all such $s$-super supports of $x$ is denoted by $\cT_s(x)$, that is, 
$\cT_s(x) = \{T \in \overline\cT_s(x): |T|=s\}$.

The sign operator $\sign: \Re^n \to \{-1,1\}^n$ 
is defined as 
\[
(\sign(x))_i = \left\{
\ba{ll}
1 &   \mbox{if} \ x_i \ge 0, \\ 
-1 & \mbox{otherwise} 
\ea \right. \quad \forall i=1,\ldots, n.
\]
The Hadmard product of any two vectors $x,y\in\Re^n$ is denoted by $x \circ y$, that is, $(x \circ y)_i = x_iy_i$ 
for $i=1,\ldots,n$. Given a closed set  $\cX \subseteq \Re^n$, the Euclidean projection of 
$x\in\Re^n$ onto $\cX$ is defined as 
the set 
\[
\proj_\cX(x) = \Arg\min\{\|y-x\|^2: y \in \cX\}.  
\]   
If $\cX$ is additionally convex, $\proj_\cX(x)$ reduces to a singleton, which is 
treated as a point by convention.  Also, the normal cone 
of $\cX$ at any $x \in \cX$ is denoted by $\cN_{\cX}(x)$.

The permutation group of the set of indices $\{1,\ldots, n\}$ is denoted by $\Sigma_n$. For any $x\in\Re^n$ and 
$\sigma \in \Sigma_n$, the vector $x^\sigma$ resulted from $\sigma$ operating on $x$ is defined as 
\[
(x^\sigma)_i = x_{\sigma(i)} \quad\quad \forall i=1,\ldots, n.
\]
Given any $x\in\Re^n$, a permutation that sorts the elements of $x$ in a non-ascending order is called a 
sorting permutation for $x$. The set of all sorting permutations for $x$ is denoted by $\tsigma(x)$. It is clear to 
see that $\sigma \in \tsigma(x)$ if and only if 
\[
x_{\sigma(1)} \ge x_{\sigma(2)} \ge \cdots \ge x_{\sigma(n-1)} \ge x_{\sigma(n)}.
\]
Given any $s \in \{1,\ldots,n-1\}$ and $\sigma \in \Sigma_n$, we define 
\[
\cS^\sigma_{[1,s]} = \{\sigma(1),\ldots, \sigma(s)\}.
\]  
A set $\cX\subseteq \Re^n$ is called a {\it symmetric set} if $x^\sigma \in \cX$ for all $x\in \cX$ and 
$\sigma\in\Sigma_n$. In addition,  $\cX$ is referred to as a {\it nonnegative symmetric set} if $\cX$ 
is a symmetric set and moreover $x \ge 0$ for all $x\in \cX$. $\cX$ is called a {\it sign-free symmetric set} 
if $\cX$ is a symmetric set and $x \circ y \in \cX$ for all $x\in \cX$ and $y\in\{-1,1\}^n$.  


\section{Projection over some sparse symmetric sets}
\label{proj-set}

In this section we study some useful properties of the orthogonal projection of a point onto the set $\cFr$, 
where $s\in\{1,\ldots,n-1\}$ is given. Throughout this paper, we make the following assumption regarding $\Omega$.

\begin{assumption} \label{assump-omega}
$\Omega$ is either a nonnegative or a sign-free symmetric closed convex set in $\Re^n$.
\end{assumption}

Let $\cp:\Re^n \to \Re^n$ be an operator associated with $\Omega$ that is 
defined as follows:
\beq \label{px}
\cp(x) = \left\{
\ba{ll}
x & \mbox{if}  \ \Omega \ \mbox{is nonnegative symmetric}, \\ 
|x| &  \mbox{if} \  \Omega\  \mbox{is sign-free symmetric}.
\ea\right.
\eeq 
One can observe that $\cp(x) \ge 0$ for all $x\in \Omega$. Moreover, $(\cp(x))_i=0$ for some $x\in\Re^n$ 
and $i\in\{1,\ldots,n\}$ if and only if $x_i=0$.

The following two lemmas were established in \cite{BeHa14}. The first one  
presents a monotone property of the orthogonal projection associated with a  
symmetric set. The second one provides a characterization of  the orthogonal projection associated with a sign-free symmetric set.

\begin{lemma}[Lemma 3.1 of \cite{BeHa14}]\label{monotone}
Let $\cX$ be a closed symmetric set in $\Re^n$. Let $x\in\Re^n$ and $y\in\proj_\cX(x)$. Then 
$(y_i-y_j) (x_i-x_j) \ge 0$  for all $i,j \in\{1,2,\ldots,n\}$.
\end{lemma}

\begin{lemma}[Lemma 3.3 of \cite{BeHa14}] \label{relation-projs}
Let $\cX$ be a closed sign-free symmetric set in $\Re^n$. Then $y \in \proj_\cX(x)$ if and only if 
$
\sign(x) \circ y  \in \proj_{\cX \cap \Re^n_+}(|x|).
$
\end{lemma}

We next establish a monotone property for the orthogonal projection operator associated with $\Omega$.

\begin{lemma} \label{monotone-1} 
Let $\cp$ be the associated operator of $\Omega$ 
defined in \eqref{px}. 
Then for every $x\in\Re^n$ and $y\in\proj_\Omega(x)$, there holds
\beq \label{p-monotone}
\left[(\cp(y))_i-(\cp(y))_j\right] \left[(\cp(x))_i-(\cp(x))_j\right] \ge 0 \quad\quad \forall i,j \in\{1,2,\ldots,n\}.
\eeq
\end{lemma}

\begin{proof}
If $\Omega$ is a closed nonnegative symmetric set, \eqref{p-monotone} clearly holds due to  \eqref{px} and 
Lemma \ref{monotone}.  Now suppose $\Omega$ is a closed sign-free symmetric set. In view of 
Lemma \ref{relation-projs} and $y\in\proj_\Omega(x)$, one can see that 
\[
\sign(x) \circ y  \in \proj_{\Omega \cap \Re^n_+}(|x|).
\]
Taking absolute value on both sides of this relation, and using the definition of 
$\sign$, we obtain that 
\beq \label{absy}
|y|  \in \proj_{\Omega \cap \Re^n_+}(|x|).
\eeq
Observe that $\Omega \cap \Re^n_+$ is a closed nonnegative symmetric set. Using this fact, \eqref{absy} 
and Lemma \ref{monotone}, we have
\[
(|y|_i -|y|_j) (|x|_i-|x|_j) \ge 0 \quad\quad \forall i,j \in\{1,2,\ldots,n\},
\] 
which, together with \eqref{px} and the fact that $\Omega$ is sign-free symmetric, 
 implies that \eqref{p-monotone} holds.
\end{proof}

\gap

The following two lemma presents some useful properties of the orthogonal 
projection operator associated with $\cFr$. The proof the first one is similar 
to that of Lemma 4.1 of \cite{BeHa14}.

\begin{lemma} \label{proj-general} 
For every $x\in\Re^n$ and $y\in\proj_{\cFr}(x)$, there holds 
\[
y_T = \proj_{\Omega_T}(x_T) \quad\quad \forall T \in \overline\cT_s(y).
\]
\end{lemma}

\begin{lemma}[Theorem 4.4 of \cite{BeHa14}] \label{pos-proj}
 Let $\cp$ be the associated operator of $\Omega$ 
defined in \eqref{px}. Then for every $x\in\Re^n$ and $\sigma \in \tsigma(\cp(x))$,  there exists 
$y\in\proj_{\cC_s \cap \Omega}(x)$ such that $\cS^\sigma_{[1,s]} \in \cT_s(y)$, that is, $\cS^\sigma_{[1,s]}$
 is a $s$-super support of $y$ with cardinality $s$.
\end{lemma}

Combining Lemmas \ref{proj-general}  and \ref{pos-proj}, we obtain the following theorem, which provides a formula 
for finding a point in $\proj_{\cFr}(x)$ for any 
$x\in\Re^n$. 

\begin{theorem} \label{opt-proj} 
Let $\cp$ be the associated operator of $\Omega$ 
defined in \eqref{px}. Given any $x\in\Re^n$, let $T=S^\sigma_{[1,s]} $ for some $\sigma \in \tsigma(\cp(x))$.  
Define $y\in\Re^n$ as follows:
\[
y_T = \proj_{\Omega_T}(x_T), \quad\quad y_{T^\c} = 0.
\]
Then $y\in\proj_{\cC_s \cap \Omega}(x)$.
\end{theorem}

In the following two theorems, we provide some sufficient conditions under which  
the orthogonal projection of a point onto $\cFr$ reduces to a single point.   

\begin{theorem} \label{unique-soln}
 Given $a\in\Re^n$, suppose there exists some $y \in \proj_{\cC_s \cap \Omega}(a)$ with $\|y\|_0 < s$. Then $\proj_{\cC_s \cap \Omega}(a)$ is 
a singleton containing $y$.
\end{theorem}

\begin{proof}
For convenience, let $I=\supp(y)$. We first show that 
\beq \label{monotone-p}
(\cp(a))_i > (\cp(a))_j \quad\quad \forall i\in I, j\in I^\c, 
\eeq
where $\cp$ is defined in \eqref{px}. By the definitions of $I$ and $\cp$, one can observe that 
$(\cp(y))_i>(\cp(y))_j$  for all $i\in I$ and $j\in I^\c$. This together with 
\eqref{p-monotone} with $x=a$ implies that 
\[
(\cp(a))_i \ge (\cp(a))_j \quad\quad \forall i\in I, j\in I^\c.
\]
It then follows that for proving \eqref{monotone-p} it suffices to show 
\[
(\cp(a))_i \neq (\cp(a))_j \quad\quad \forall i\in I, j\in I^\c.
\]
Suppose on the contrary that $(\cp(a))_i = (\cp(a))_j$ for some $i\in I$ and $j\in I^\c$. Let
\[
\beta = \left\{
\ba{ll}
y_i & \mbox{if} \  \Omega \ \mbox{is nonnegative symmetric}, \\ 
\sign(a_j) |y_i|  &  \mbox{if} \  \Omega\  \mbox{is sign-free symmetric},
\ea\right.
\]
and let $z\in\Re^n$ be defined as follows:
\[
z_\ell = \left\{
\ba{ll}
y_j  & \ \mbox{if} \ \ell=i ,  \\
\beta & \ \mbox{if} \  \ell=j ,  \\
y_\ell & \ \mbox{otherwise}, 
\ea
\right.
\quad\quad \ell =1, \ldots, n.
\]
Since $\Omega$ is either nonnegative  or sign-free symmetric, it is not hard to see that $z\in\Omega$. Notice 
that $y_i \neq 0$ and $y_j=0$ due to $i\in I$ and $j\in I^\c$. This together with $\|y\|_0<s$ and 
the definition of $z$ implies $\|z\|_0<s$. Hence, $z\in\cC_s \cap \Omega$. In view of Lemma \ref{relation-projs} 
with $x=a$ and $\cX=\Omega$ and $y \in \proj_{\cC_s \cap \Omega}(a)$, one can observe that $y \circ a \ge 0$ when 
$\Omega$ is sign-free symmetric. Using this fact and the definitions of 
$\cp$ and $z$, one can observe that 
\[
y_i a_i = (\cp(y))_i  (\cp(a))_i, \quad\quad z_j a_j = (\cp(y))_i  (\cp(a))_j,
\]
which along with the supposition $(\cp(a))_i = (\cp(a))_j$ yields $y_i a_i = z_j a_j$. In addition, 
one can see that $z_j^2 = y_i^2$. Using these two relations,  $z_i=y_j=0$, and the definition of 
$z$, we have   
\beqa
\|z-a\|^2 &=& \sum\limits_{\ell \neq i, j} (z_\ell-a_\ell)^2 + (z_i-a_i)^2 + (z_j-a_j)^2  
= \sum\limits_{\ell \neq i, j} (y_\ell-a_\ell)^2 + a^2_i + z_j^2 - 2 z_ja_j +a^2_j \nn \\ 
&=& \sum\limits_{\ell \neq i, j} (y_\ell-a_\ell)^2 + a^2_i + y_i^2 - 2 y_ia_i +a^2_j  
= \|y-a\|^2. \label{y-opt}
\eeqa
 In addition, by the definition of $z$ and the convexity of $\Omega$, 
it is not hard to observe that  $y \neq z$ and $(y+z)/2 \in \cC_s \cap \Omega$.  By the strict convexity 
of $\|\cdot\|^2$, $y\neq z$ and \eqref{y-opt}, one has
\[
\left\|\frac{y+z}{2}-a\right\|^2 < \frac 12 \|y-a\|^2 +  \frac 12 \|z-a\|^2 = \|y-a\|^2,
\] 
which together with $(y+z)/2 \in \cC_s \cap \Omega$ contradicts the assumption 
$y \in \proj_{\cC_s \cap \Omega}(a)$. Hence, \eqref{monotone-p} holds as desired. 

We next show that for any $z\in \proj_{\cC_s \cap \Omega}(a)$, it holds that $\supp(z) \subseteq I$, 
where $I=\supp(y)$. Suppose for contradiction that there exists some $j\in I^\c$ such that $z_j \neq 0$, 
which together with the definition of $\cp$ yields $(\cp(z))_j \neq 0$.  Clearly,  $\cp(z) \ge 0$ due to $z\in\Omega$ 
and \eqref{px}.  It then follows that $(\cp(z))_j >0$.  In view of Lemma \ref{monotone-1}, we further have
\[
 [(\cp(z))_i - (\cp(z))_j][(\cp(a))_i - (\cp(a))_j]  \ge 0  \quad\quad \forall i\in I,
\]
which together with $j\in I^\c$ and \eqref{monotone-p}  implies $(\cp(z))_i \ge  (\cp(z))_j$ for all $i\in I$.  Using this, 
$(\cp(z))_j >0$ and the definition of $\cp$, we see that $(\cp(z))_i>0$ and hence $z_i \neq 0$ for all 
$i\in I$. Using this relation, $y,z\in\Omega$, and the convexity of $\Omega$, one can see that $(y+z)/2 
\in \cC_s \cap \Omega$. In addition, since $y, z\in \proj_{\cC_s \cap \Omega}(a)$, we have 
$\|y-a\|^2=\|z-a\|^2$. Using this and a similar argument as above, one can show that 
$\|(y+z)/2-a\|^2 <  \|y-a\|^2$, which together with $(y+z)/2 \in \cC_s \cap \Omega$ 
contradicts the assumption $y\in \proj_{\cC_s \cap \Omega}(a)$.  

Let $z \in \proj_{\cC_s \cap \Omega}(a)$. As shown above, $\supp(z) \subseteq \supp(y)$. Let $T \in 
\cT_s(y)$. It then follows that $T\in\cT_s(z)$. Using these two relations and Lemma \ref{proj-general}, 
we have $y_T = \proj_{\Omega_T}(a_T) = z_T$. 
Notice that $y_{T^\c}=z_{T^\c}=0$. It thus follows $y=z$, which implies that the set 
$\proj_{\cC_s \cap \Omega}(a)$  contains $y$ only.
\end{proof}
\gap
\begin{theorem} \label{unique-soln1}
Given $a\in\Re^n$, suppose there exists some $y \in \proj_{\cC_s \cap \Omega}(a)$ such that 
\beq \label{a-ineq}
\min\limits_{i \in I} \ (\cp(a))_i > \max\limits_{i \in I^\c} \ (\cp(a))_i, 
\eeq
where $I=\supp(y)$. Then $\proj_{\cC_s \cap \Omega}(a)$ is a singleton containing $y$.
\end{theorem}

\begin{proof}
We divide the proof into two separate cases as follows. 

Case 1):  $\|y\|_0<s$.  The conclusion holds due to Theorem \ref{unique-soln}.

Case 2): $\|y\|_0=s$. This along with $I=\supp(y)$ yields $|I|=s$. Let $z\in\proj_{\cC_s \cap \Omega}(a)$.  
In view of Lemma \ref{monotone-1} and 
\eqref{a-ineq}, one has 
\beq \label{py-ineq}
\min\limits_{i \in I} \ (\cp(z))_i \ge \max\limits_{i \in I^\c} \ (\cp(z))_i.
\eeq
Notice $z\in\cC_s \cap \Omega$. Using this and the definition of $\cp$,  we observe that 
$\|\cp(z)\|_0=\|z\|_0  \le s$ and $\cp(z) \ge 0$. These relations together with $|I|=s$ and 
\eqref{py-ineq} imply that $(\cp(z))_i=0$ for all $i \in I^\c$. This yields $z_{I^\c}=0$. 
It then follows that $\supp(z) \subseteq I$. Hence, $I \in \cT_s(z)$  and $I \in \cT_s(y)$ due to $|I|=s$. 
Using these, Lemma \ref{proj-general}, and $y,z\in\proj_{\cC_s \cap \Omega}(a)$, one has 
\[
z_I = \proj_{\Omega_I}(a_I) = y_I,
\]
which together with $z_{I^\c}=y_{I^\c}=0$ implies $z=y$. Thus $\proj_{\cC_s \cap \Omega}(a)$ contains only $y$.
\end{proof}


\section{Optimality conditions} \label{opt-condns}

In this section we study some optimality conditions for problem \eqref{sparse-prob}. We start by reviewing a necessary optimality condition that was established in Theorem 5.3 of \cite{BeHa14}.

\begin{theorem}[necessary optimality condition]\label{nec-cond}
Suppose that $x^*$ is an optimal solution of problem \eqref{sparse-prob}. Then there holds
\beq \label{nec-cond0}
x^*\in \proj_\cFr(x^*-t\nabla f(x^*)) \quad\quad \forall t\in[0,1/L_f),
\eeq 
 that is, $x^*$ is an optimal (but possibly not unique) solution to the problems 
\[
\min\limits_{x\in \cFr} \|x-(x^*-t\nabla f(x^*))\|^2 \quad\quad \forall t\in[0,1/L_f).
\]
\end{theorem}


We next establish a stronger necessary optimality condition than the one stated 
above for \eqref{sparse-prob}.

\begin{theorem}[strong necessary optimality condition]  \label{opt-cond-prop1}
Suppose that $x^*$ is an optimal solution of problem \eqref{sparse-prob}. Then there holds
\beq \label{opt-cond1}
x^*= \proj_\cFr(x^*-t\nabla f(x^*)) \quad\quad \forall t\in[0,1/L_f),
\eeq
that is, $x^*$ is the unique optimal solution to the problems 
\beq \label{gopt-soln} 
\min\limits_{x\in \cFr} \|x-(x^*-t\nabla f(x^*))\|^2 \quad\quad \forall t\in[0,1/L_f).
\eeq 
\end{theorem}

\begin{proof}
The conclusion clearly holds for $t=0$.  Now let $t \in (0,1/L_f)$ be arbitrarily chosen. Suppose for contradiction 
that problem \eqref{gopt-soln} has an optimal solution $\tx^* \in \cFr$ with $\tx^*\neq x^*$. It then follows that    
\[
\|\tx^*-(x^*-t\nabla f(x^*))\|^2 \le \|x^*-(x^*-t\nabla f(x^*))\|^2,
\]
which leads to 
\[
\nabla f(x^*)^T(\tx^*-x^*) \le -\frac{1}{2t} \|\tx^*-x^*\|^2.
\]
In view of this relation, \eqref{lipschitz} and the facts that $t \in (0,1/L_f)$ and $\tx^* \neq x^*$, one 
can obtain that 
\beqa
f(\tx^*) &\le &  f(x^*)+\nabla f(x^*)^T(\tx^*-x^*) + \frac{L_f}{2}\|\tx^*-x^*\|^2 \nn \\ 
&\le &  f(x^*) + \frac12\left(L_f-\frac1t\right)\|\tx^*-x^*\|^2 \ < \ f(x^*), \label{fw}
\eeqa
which contradicts the assumption that $x^*$ is an optimal solution of  problem \eqref{sparse-prob}.
\end{proof}

\gap

For ease of later reference, we introduce the following definitions.

\begin{definition} 
$x^*\in \Re^n$ is called a general stationary point of problem \eqref{sparse-prob} if it satisfies the necessary 
optimality condition \eqref{nec-cond0}. 
\end{definition} 

\begin{definition} 
$x^*\in \Re^n$ is called a  strong stationary point of problem \eqref{sparse-prob} if it satisfies the strong 
necessary optimality condition \eqref{opt-cond1}. 
\end{definition} 

Clearly, a  strong stationary point must be a general stationary point, but the converse may not be true. In 
addition,  from the proof of Theorem \ref{opt-cond-prop1}, one can easily improve the quality of  a general but 
not a strong stationary point $x^*$ by finding a point 
$\tx^*\in\proj_\cFr(x^*-t\nabla f(x^*))$ with $\tx^*\neq x^*$.  As seen from above, $f(\tx^*) < f(x^*)$, which 
means $\tx^*$ has a better quality than $x^*$ in terms of the objective value of \eqref{sparse-prob}.

Before ending this section, we present another necessary optimality condition for problem 
\eqref{sparse-prob} that was established in \cite{BeHa14}.

\begin{theorem}[Lemma 6.1 of \cite{BeHa14}]\label{thm-beck}
If $x^*$ is an optimal solution of  problem \eqref{sparse-prob}, there hold: 
\[
\ba{l}
x^*_T = \proj_{\Omega_T}(x^*_T-t(\nabla f(x^*))_T) \quad\quad \forall T\in\cT_s(x^*),  \\ [8pt]
f(x^*) \le \left\{\ba{ll} 
\min\{f(x^*-x^*_i \be_i+x^*_i \be_j), f(x^*-x^*_i \be_i-x^*_i \be_j)\} & \mbox{if} \ \Omega \ 
\mbox{is sign-free} \\ 
&\mbox{symmetric};  \\ [8pt]
f(x^*-x^*_i \be_i+x^*_i \be_j) & \mbox{if} \ \Omega \ 
\mbox{is  nonnegative} \\ 
&\mbox{symmetric}
\ea\right.
\ea
\]
for some $t>0$ and some $i,j$ satisfying 
\[
\ba{l}
i \in \Arg\min\{(\cp(-\nabla f(x^*)))_\ell: \ell \in I\}, \\ [5pt]
j \in \Arg\max\{(\cp(-\nabla f(x^*)))_\ell: \ell \in [\supp(x^*)]^\c\},
\ea
\]
 where $I= \Arg\min\limits_{i \in \supp(x^*)} (\cp(x^*))_i$.
\end{theorem}

\begin{definition} 
$x^*\in \Re^n$ is called a coordinatewise stationary point of problem \eqref{sparse-prob} if it 
satisfies the necessary optimality condition stated in Theorem \ref{thm-beck}.
\end{definition}


%

\section{A nonmonotone projected gradient method}
\label{method}


As seen from Theorem \ref{opt-proj}, the orthogonal projection a point onto $\cFr$  can be efficiently computed. 
Therefore, the classical projected gradient (PG) method with a constant step size can be suitably applied to solve problem 
\eqref{sparse-prob}.
In particular, given a $\T \in (0,1/L_f)$ and $x^0 \in \cFr$,  the PG method generates a sequence 
$\{x^k\} \subseteq \cFr$ according to 
\beq \label{x-seq}
x^{k+1} \in \proj_\cFr\left(x^k-\T\nabla f(x^k)\right) \quad\quad \forall k \ge 0.
\eeq 
The iterative hard-thresholding (IHT) algorithms \cite{BlDa08,BlDa09} are 
either a special case or a variant of the above method. 

By a similar argument as in the proof of Theorem \ref{opt-cond-prop1}, one can show that 
\beq \label{f-monotone}
f(x^{k+1}) \le f(x^k) - \frac12\left(\frac{1}{\T}-L_f\right) \|x^{k+1}-x^k\|^2 \quad\quad \forall k \ge 0,
\eeq
which implies that $\{f(x^k)\}$ is non-increasing. Suppose that $x^*$ is an accumulation point $\{x^k\}$. 
It then follows that $f(x^k) \to f(x^*)$.  In view of this and \eqref{f-monotone}, one can see that 
$\|x^{k+1}-x^k\| \to 0$, which along with \eqref{x-seq} and \cite[Theorem 5.2]{BeHa14} yields
\beq \label{T-station}
x^*\in \proj_\cFr(x^*-t\nabla f(x^*)) \quad\quad \forall t \in [0, \T].
\eeq
Therefore,  when $\T$ is close to $1/L_f$,  $x^*$  is nearly a general stationary point of problem 
\eqref{sparse-prob}, that is, it nearly satisfies the necessary optimality condition \eqref{nec-cond0}.
 It is, however, still possible that there exists some $\hatt \in (0,\T]$ such that 
$
x^* \neq \proj_\cFr(x^*-\hatt\nabla f(x^*)).
$
As discussed in Section \ref{opt-condns}, in this case one has $f(\tx^*)<f(x^*)$ for 
any $\tx^*\in\proj_\cFr(x^*-\hatt\nabla f(x^*))$ with $\tx^*\neq x^*$, and thus $x^*$ is clearly not 
an optimal solution of problem \eqref{sparse-prob}. To prevent this case from occuring, 
we propose a nonmonotone projected gradient (NPG) method for solving problem 
\eqref{sparse-prob}, which incorporates some support-changing and coordinate-swapping strategies into a 
projected gradient approach with variable stepsizes. 
We show that any accumulation point $x^*$ of  the sequence generated by NPG  
satisfies 
\[
x^*= \proj_\cFr(x^*-t\nabla f(x^*)) \quad\quad \forall t \in [0, \T]
\]
for any pre-chosen $\T \in (0,1/L_f)$. Therefore,  when $\T$ is close to $1/L_f$,  $x^*$ 
is nearly a strong stationary point of problem \eqref{sparse-prob}, that is, it 
nearly satisfies the strong necessary optimality condition \eqref{opt-cond1}.  
Under some suitable assumptions, we further show 
that $x^*$ is a  coordinatewise stationary point. Furthermore, if $\|x^*\|_0<s$,  
then $x^*$ is an optimal solution of \eqref{sparse-prob}, and it is a local optimal 
solution otherwise.

\subsection{Algorithm framework of NPG} \label{NPG}

In this subsection we present an NPG method for solving problem \eqref{sparse-prob}.  To proceed, we first 
introduce a subroutine called $\cswap$ that generates 
a new point $y$ by swapping some coordinates of a given point $x$. 
The aim of this subroutine is to generate a new point $y$ with a smaller objective value, namely, 
$f(y)<f(x)$, if the given point $x$ violates the second part of  the coordinatewise optimality conditions 
stated in Theorem \ref{thm-beck}. Upon incorporating this subroutine into the NPG method, we show 
that under some suitable assumption, any accumulation point of the generated sequence is a 
coordinatewise stationary point of  \eqref{sparse-prob}.

\gap

\noindent{\bf The subroutine $\cswap(x)$} \\ [4pt]
{\bf Input:} $x\in\Re^n$.
\bi
 \item[1)] Set $y=x$ and choose 
\[
\ba{l}
i \in \Arg\min\{(\cp(-\nabla f(x)))_\ell: \ell \in I\}, \\ [5pt]
j \in \Arg\max\{(\cp(-\nabla f(x)))_\ell: \ell \in [\supp(x)]^\c\},
\ea
\]
where $I= \Arg\min\{(\cp(x))_i: i \in \supp(x)\}$.
\item[2)] If $\Omega$ is nonnegative symmetric and $f(x)>f(x-x_i \be_i+x_i \be_j)$, set $y=x-x_i \be_i+x_i \be_j$.
\item[3)] If  $\Omega$ is sign-free symmetric and $f(x)> \min\{f(x-x_i \be_i+x_i \be_j), f(x-x_i \be_i-x_i \be_j)\}$, 
\bi 
\item[3a)] if $f(x-x_i \be_i+x_i \be_j) \le  f(x-x_i \be_i-x_i \be_j)$, set $y=x-x_i \be_i+x_i \be_j$. 
\item[3b)] if $f(x-x_i \be_i+x_i \be_j) >  f(x-x_i \be_i-x_i \be_j)$, set $y=x-x_i \be_i-x_i \be_j$. 
\ei
\ei
{\bf Output:} $y$.

\gap

We next introduce another subroutine called $\cs$ that generates a new point by changing  
some part of the support of a given point. Upon incorporating this subroutine into the NPG method, 
we show that  any accumulation point of the generated sequence is nearly a strong stationary point 
of \eqref{sparse-prob}.

\gap

 \noindent{\bf The subroutine $\cs(x,t)$} \\ [4pt]
{\bf Input:} $x\in\Re^n$ and $t\in\Re$.
\bi
 \item[1)] Set $a=x-t\nabla f(x)$, $I= \Arg\min\limits_{i \in \supp(x)} (\cp(a))_i$, and 
$J= \Arg\max\limits_{j \in [\supp(x)]^\c} (\cp(a))_j$.
\item[2)] Choose $S_I \subseteq I$ and $S_J \subseteq J$ such that $|S_I|=|S_J|=\min\{|I|,|J|\}$. 
Set $S=\supp(x)\cup S_J\setminus S_I$.
\item[3)] Set $y \in \Re^n$ with $y_S = \proj_{\Omega_S}(a_S)$ and $y_{S^\c}=0$.
\ei
{\bf Output:} $y$.

\gap

One can observe that if $0<\|x\|_0<n$, the output $y$ of $\cs(x,t)$ must satisfy 
$\supp(y) \neq \supp(x)$ and thus $y \neq x$. We next introduce some notations that will be used 
subsequently.  

Given any $x\in\Re^n$ with $0<\|x\|_0<n$ and $\T>0$, define 
\beqa 
& \gamma(t;x) =  \min\limits_{i \in \supp(x)} \left(\cp(x-t\nabla f(x))\right)_i - 
\max\limits_{j \in [\supp(x)]^\c} \left(\cp(x-t\nabla f(x))\right)_j \quad \forall t \ge 0,  \label{gt}  \\
& \beta(\T;x) \in \Arg\min\limits_{t\in [0,\T]} \gamma(t;x), \quad\quad \vartheta(\T;x) = \min\limits_{t\in [0,\T]} \gamma(t;x). \label{bT}
\eeqa
If $x=0$ or $\|x\|_0=n$, define $\beta(\T;x)=\T$ and $\vartheta(\T;x)=0$. In addition, 
if the optimization problem \eqref{bT} has multiple optimal solutions, $\beta(\T;x)$ is 
 chosen to be the largest one among them. As seen below, $\beta(\T;x)$ and $\vartheta(\T;x)$ 
can be evaluated efficiently.

To avoid triviality, assume $0<\|x\|_0<n$.  It follows from \eqref{px} and \eqref{gt} that 
\beq \label{gtx}
\gamma(t;x) =  \min\limits_{i \in \supp(x)} \phi_i(t;x) \quad\quad \forall t \ge 0,
\eeq
where 
\[
\phi_i(t;x) = \left(\cp(x-t\nabla f(x))\right)_i - \alpha t, \quad\quad 
\alpha = \max\limits_{j \in [\supp(x)]^\c} \left(\cp(-\nabla f(x))\right)_j.
\]
We now consider two separate cases as follows.

Case 1): $\Omega$ is nonnegative symmetric. In view of \eqref{px}, we see that in this case 
\[
\phi_i(t;x) = x_i- \left(\frac{\partial f}{\partial x_i}  + \alpha\right) t \quad\quad \forall i.
\]
This together with \eqref{gtx} implies that $\gamma(t;x)$ is concave with respect to $t$. Thus, 
the minimum value of $\gamma(t;x)$ for $t\in[0,\T]$ must be achieved at $0$ or $\T$. It then 
follows that $\beta(\T;x)$ and $\vartheta(\T;x)$ can be found by comparing $\gamma(0;x)$ and 
$\gamma(\T;x)$, which can be evaluated in $O(\|x\|_0)$ cost. 
 
Case 2): $\Omega$ is sign-free symmetric. In view of  \eqref{bT} and \eqref{gtx}, one can observe that
\[
\vartheta(\T;x) = \min\limits_{t\in [0,\T]} \left\{ \min\limits_{i \in \supp(x)} \phi_i(t;x) \right\} 
= \min\limits_{i \in \supp(x)} \left\{ \min\limits_{t\in [0,\T]} \phi_i(t;x) \right\}.
\]
By the definition of $\cp$, it is not hard to see that $\phi_i(\cdot;x)$ is a convex 
piecewise linear function of 
$t\in (-\infty,\infty)$. Therefore, for a given $x$, one can find a closed-form expression for 
\[
\phi^*_i(x) =  \min\limits_{t\in [0,\T]} \phi_i(t;x), \quad\quad t^*_i(x) \in  \Arg\min\limits_{t\in [0,\T]} \phi_i(t;x).
\] 
Moreover, their associated arithmetic operation cost is $O(1)$ for each $x$. 
Let $i^*\in \supp(x)$ be such that 
\[
 \phi^*_{i^*}(x) = \min\limits_{i \in \supp(x)} \phi^*_i(x).
\]
Then  we obtain $\vartheta(\T;x) =  \phi^*_{i^*}(x)$ and $\beta(\T;x) = t^*_{i^*}(x)$. Therefore, for a given $x$, 
$\vartheta(\T;x)$ and $\beta(\T;x)$ can be computed in $O(\|x\|_0)$ cost.

Combining the above two cases, we reach the following conclusion regarding $\beta(\T;x)$ and $\vartheta(\T;x)$.

\begin{proposition}
For any $x$ and $\T>0$, $\beta(\T;x)$ and $\vartheta(\T;x)$ can be computed in $O(\|x\|_0)$ cost. 
\end{proposition}

We  are now ready to present an NPG method for solving problem \eqref{sparse-prob}.

\gap

\noindent
{\bf Nonmonotone projected gradient (NPG) method for \eqref{sparse-prob}}  \\ [5pt]
Let $0<t_{\min}<t_{\max}$, $\tau \in (0,1)$,  $\T\in (0, 1/L_f)$, $c_1 \in (0, 1/\T-L_f)$, $c_2>0$, $\eta>0$, 
and integers $N \ge 3$, $0\le M <N$, $0<\q<N$ be given. Choose an arbitrary $x^0 \in \cC_s \cap \Omega$ and set $k=0$.
\begin{itemize}
\item[1)] ({\bf coordinate swap}) If $\Mod(k,N)=0$, do 
\bi
\item[1a)] Compute $\bx^{k+1}=\cswap(x^k,\T)$. 
\item[1b)] If $\bx^{k+1} \neq x^k$, set $x^{k+1}=\bx^{k+1}$ and go to step 4). 
\item[1c)] If $\bx^{k+1} = x^k$, go to step 3). 
\ei
\item[2)] ({\bf change support}) If $\Mod(k,N)=\q$ and $\vartheta(\T;x^k)\le\eta$, do
\bi
\item[2a)] Compute 
\beqa 
\tx^{k+1} &\in& \proj_{\cC_s \cap \Omega}\left(x^k-\beta(\T;x^k)\nabla f(x^k)\right), \label{tx} \\ 
\hx^{k+1}&=& \cs\left(\tx^{k+1},\beta(\T;x^k)\right). \label{hx}
\eeqa
\item[2b)] If 
\beq \label{descent0}
f(\hx^{k+1}) \le f(\tx^{k+1})-\frac{c_1}{2}\|\hx^{k+1}-\tx^{k+1}\|^2
\eeq
holds, set $x^{k+1}=\hx^{k+1}$ and go to step 4).
\item[2c)] If $\beta(\T;x^k)>0$, set $x^{k+1}=\tx^{k+1}$ and go to step 4).
\item[2d)] If $\beta(\T;x^k)=0$, go to step 3).
\ei 
\item[3)] ({\bf projected gradient}) Choose $t^0_k \in [t_{\min}, t_{\max}]$. Set $\bt_k = t^0_k$.
\bi
\item[3a)] Solve the subproblem
\beq \label{subprob}
w \in \proj_{\cC_s \cap \Omega} (x^k-\bt_k\nabla f(x^k)),
\eeq
\item[3b)] If
\beq \label{descent}
f(w) \le \max\limits_{[k-M]_+ \le i \le k} f(x^i) - \frac{c_2}{2} \|w-x^k\|^2
\eeq
holds, set $x^{k+1}=w$, $t_k=\bt_k$, and go to step 4).
\item[3c)] Set $\bt_k \leftarrow \tau \bt_k$ and go to step 3a).
\ei
\item[4)]
Set $k \leftarrow k+1$, and go to step 1).
\end{itemize}
\noindent
{\bf end}

\gap

{\bf Remark:}
\begin{itemize}
\item[(i)]
When $M=0$, the sequence $\{f(x^k)\}$ is decreasing. Otherwise, it may increase at some 
iterations and thus the above method is generally a nonmonotone method.
\item[(ii)] A popular choice of $t^0_k$ is by  the following formula proposed by Barzilai and Borwein 
\cite{BaBo88}:
\[
t^0_k  = \left\{\ba{ll}
\min \left\{t_{\max } ,\max\left\{t_{\min},\frac{\|s^k\|^2}{|(s^k)^T y^k|}\right\}\right\}, &  \ \mbox{if} \ (s^k)^T y^k \neq 0, \\
t_{\max} & \ \mbox{otherwise}.
\ea\right.
\]
where $s^k  = x^k  - x^{k - 1}$, $y^k=\nabla f(x^k)-\nabla f(x^{k - 1})$.
\end{itemize}

\subsection{Convergence results of NPG} \label{conv-results}

In this subsection we study convergence properties of the NPG method proposed in Subsection \ref{NPG}. We first state that the inner termination criterion \eqref{descent} is satisfied in a certain number of inner iterations, whose proof is similar to \cite[Theorem 2.1]{XuLuXu15}.

\begin{theorem} \label{inner-convergence}
The inner termination criterion \eqref{descent} is satisfied after at most
\[
\max\left\{\left\lfloor -\frac{\log(L_f+c_2)+\log(t_{\max})}{\log \tau} +2\right\rfloor,1\right\}
\]
inner iterations, and moreover, 
\beq \label{tk}
 \min \left\{t_{\min},  \tau/(L_f+c_2)\right\}  \ \le\  t_k \ \le \  t_{\max},
\eeq 
where $t_k$ is defined in Step 2) of the NPG method.
\end{theorem}

In what follows, we study the convergence of the outer iterations of the NPG method. Throughout the 
rest of this subsection, we make the following assumption regarding $f$. 

\begin{assumption} \label{assump-f}
$f$ is bounded below in $\cC_s \cap \Omega$, and moreover it is uniformly continuous in the 
level set 
\[
\Omega_0 := \{x\in \cC_s \cap \Omega: f(x) \le f(x^0)\}.
\]  
\end{assumption}

We start by establishing a convergence result regarding the sequences 
$\{f(x^k)\}$ and $\{\|x^{k}-x^{k-1}\|\}$. A similar result was established in \cite[Lemma 4]{WrNoFi09} 
for a nonmonotone proximal gradient method for solving a class of optimization problems. 
Its proof substantially relies on the relation: 
\[
\phi(x^{k+1}) \le \max\limits_{[k-M]_+ \le i \le k} \phi(x^i) - \frac{c}{2} \|x^{k+1}-x^k\|^2 \quad\quad \forall k \ge 0
\]
for some constant $c>0$, where $\phi$ is the associated objective function 
of the optimization problem considered in \cite{WrNoFi09}. Notice that for our NPG method,  this type of inequality  
holds only for a subset of indices $k$. Therefore, the proof of \cite[Lemma 4]{WrNoFi09} is 
not directly applicable here and a new proof is required. To make our presentation smooth, we leave the proof 
of the following result in Subsection \ref{proof-main}. 

\begin{theorem} \label{converge-fxk}
Let $\{x^k\}$ be the sequence generated by the NPG method and 
\beq \label{cN}
\cN = \{k: x^{k+1} \ \mbox{is generated by step 2) or 3) of NPG}\}.
\eeq 
There hold:
\bi
\item[(i)] $\{f(x^k)\}$ converges as $k \to \infty$; 
\item[(ii)] $\{\|x^{k+1}-x^k\|\} \to 0$ as $k\in \cN \to \infty$.
\ei
\end{theorem}

The following theorem shows that any accumulation point of $\{x^k\}$ generated by the NPG method is 
nearly a strong stationary point of problem \eqref{sparse-prob}, that is, it nearly satisfies the strong necessary 
optimality condition \eqref{opt-cond1}. Since its proof is quite lengthy and technically involved, we present it in 
Subsection \ref{proof-main} instead.

\begin{theorem} \label{lim-point}
Let $\{x^k\}$ be the sequence generated by the NPG method. Suppose that $x^*$ is an accumulation 
point of $\{x^k\}$. Then there hold:
\begin{itemize}
\item[(i)] if $\|x^*\|_0 <s$, there exists $\hatt \in [\min\{t_{\min}, \tau/(L_f+c_2)\}, t_{\max}]$ such that 
$x^*=\proj_\cFr(x^*-t\nabla f(x^*))$ for all $t\in [0,\hatt]$, that is, $x^*$ is the unique optimal solution 
to the problems
\[
\min\limits_{x\in\cFr} \|x-(x^*-t\nabla f(x^*))\|^2 \quad\quad \forall t\in [0,\hatt];
\]
\item[(ii)] if $\|x^*\|_0=s$, then $x^*=\proj_\cFr(x^*-t\nabla f(x^*))$ for all $t\in [0,\T]$, that is, 
$x^*$ is the unique optimal solution to the problems
\[
\min\limits_{x\in\cFr} \|x-(x^*-t\nabla f(x^*))\|^2 \quad\quad \forall t\in [0,\T];
\]
\item[(iii)] if $t_{\min}$, $\tau$, $c_2$ and $\T$ are chosen such that 
\beq \label{special-param}
\min\{t_{\min}, \tau/(L_f+c_2)\} \ \ge \ \T,
\eeq 
then $x^*=\proj_\cFr(x^*-t\nabla f(x^*))$ for all $t\in [0,\T]$;
\item[(iv)] if $\|x^*\|_0=s$ and $f$ is additionally convex in $\Omega$, then $x^*$ is a local optimal 
solution of problem \eqref{sparse-prob};
\item[(v)] if $\|x^*\|_0=s$ and $(\cp(x^*) )_i\neq (\cp(x^*))_j$ for all $i \neq j\in\supp(x^*)$,  then 
$x^*$ is a coordinatewise stationary point of problem \eqref{sparse-prob}.
\end{itemize}
\end{theorem}

%


\gap

Before ending this subsection, we will establish some stronger results than those given in Theorem 
\ref{lim-point}  under an additional assumption regarding  $\Omega$ that is stated below. 

\begin{assumption} \label{global-assump}
Given any $x\in \Omega$ with $\|x\|_0<s$,  if $v\in\Re^n$ satisfies 
\beq \label{vT1}
v_T \in \cN_{\Omega_T}(x_T) \quad\quad \forall T\in \cT_s(x),
\eeq 
then $v\in\cN_\Omega(x)$.
\end{assumption}

The following result shows that Assumption \ref{global-assump} holds for some widely used sets $\Omega$. 

\begin{proposition} \label{assump2-prop}
Suppose that $\cX_i$,  $i=1,\ldots,n$, are closed intervals in $\Re$, $a\in\Re^n$ with $a_i \neq 0$ for all $i$, 
$b\in\Re$, and $g:\Re^n_+ \to \Re$ is a smooth increasing convex. Let 
\[
\cB = \cX_1 \times \cdots \times \cX_n,   \quad\quad  \cC=\{x: a^Tx -b =0\}, \quad\quad \cQ = \{x: g(|x|)\le 0\}.
\]
 Suppose additionally that $g(0)<0$, and moreover, $[\nabla g(x)]_{\supp(x)} \neq 0$ for any $0 \neq x\in\Re^n_+$. 
 Then  Assumption \ref{global-assump} holds for $\Omega = \cB,  \ \cC, \  \cQ,  \  \cC\cap \Re^n_+,\  \cQ \cap \Re^n_+$, respectively.
\end{proposition}

\begin{proof}
We only prove the case where $\Omega=\cQ \cap \Re^n_+$ since the other cases can be similarly proved. To this 
end, suppose that $\Omega=\cQ \cap \Re^n_+$, $x\in\Omega$ with $\|x_0\|<s$, and $v\in\Re^n$ satisfies 
\eqref{vT1}. We now prove $v\in\cN_\Omega(x)$ 
by considering two separate cases as follows.

Case 1): $g(|x|)<0$. It then follows that $\cN_{\Omega}(x) = \cN_{\Re^n_+}(x)$ and 
\[
\cN_{\Omega_T}(x_T) = \cN_{\Re^s_+}(x_T) \quad\quad \forall T\in\cT_s(x).
\]
Using this and the assumption that $v$ satisfies  \eqref{vT1}, one has  $v_T \in \cN_{\Re^s_+}(x_T)$ for
 all $T\in\cT_s(x)$, which implies $v\in \cN_{\Re^n_+}(x)=\cN_{\Omega}(x)$. 

Case 2): $g(|x|)=0$. This together with $g(0)<0$ implies $x \neq 0$. Since $g$ is a smooth increasing convex in  
$\Re^n_+$, one can show that 
\beq \label{pg}
\partial g(|x|) = \nabla g(|x|) \circ (\partial |x_1|, \ldots, \partial |x_n|)^T.
\eeq 
In addition, by $g(0)<0$ and \cite[Theorems 3.18, 3.20]{Rus06}, one has 
\beqa
\cN_\Omega(x) &=& \{\alpha u: \alpha \ge 0, u \in \partial g(|x|)\} + \cN_{\Re^n_+}(x), \label{cN1}\\ 
\cN_{\Omega_T}(x_T) &=& \{\alpha u_T: \alpha \ge 0, u \in \partial g(|x|)\} + \cN_{\Re^s_+}(x_T) 
\quad\quad \forall T\in\cT_s(x). \label{cN2}  
\eeqa
Recall that $0 <\|x\|_0<s$. Let $I=\supp(x)$. 
It follows that $I \neq \emptyset$, and moreover, for any $j \in I^\c$, there exists some $T_j \in\cT_s(x)$ such that $j\in T_j$.  
Since $v$ satisfies \eqref{vT1}, one can observe from \eqref{pg} and \eqref{cN2} that for any $j \in I^\c$, 
there exists some $\alpha_j \ge 0$, $h^j \in  \cN_{\Re^s_+}(x_{T_j})$ and $w^j\in(\partial |x_1|, \ldots, \partial |x_n|)^T$ such that 
\[
v_{T_j}=\alpha_j [\nabla g(|x|)]_{T_j} \circ w^j_{T_j} + h^j.
\]
Using this,  $I=\supp(x) \subset T_j$, $x_I>0$ and $x_{I^\c}=0$, we see that 
\beq \label{vTj}
v_I=\alpha_j [\nabla g(|x|)]_I, \quad v_j \in t_j(\nabla g(|x|))_j[-1,1] + q_j \quad\quad \forall j \in I^\c
\eeq 
for some $q_j \le 0$ with $j \in I^\c$.  
Since $x\neq 0$ and $I=\supp(x)$, we have from the assumption that $[\nabla g(|x|)]_I \neq 0$. This along with 
\eqref{vTj} implies that there exists some $\alpha\ge 0$ such that $\alpha_j=\alpha$ for all $ j \in I^\c$. It then follows 
from this, $x_{I^\c}=0$ and \eqref{vTj} that 
\[
v_I = \alpha [\nabla g(|x|)]_I, \quad\quad v_{I^\c} \in \alpha[\nabla g(|x|)]_{I^\c}[-1,1] + \cN_{\Re^{n-|I|}_+}\left(x_{I^\c}\right),
\]
which together with \eqref{pg}, \eqref{cN1} and $x_I>0$ implies that $v\in\cN_\Omega(x)$.
\end{proof}

\gap

As an immediate consequence of Proposition \ref{assump2-prop},  Assumption \ref{global-assump} holds for some sets $\Omega$ that were recently considered in \cite{BeHa14}. 

\begin{corollary} \label{special-sets}
Assumption \ref{global-assump} holds for $\Omega=\Re^n$, $\Re^n_+$, $\Delta$, $\Delta_+$, $\cB^p(0;r)$  
and $\cB^p_+(0;r)$, where 
\[
\ba{ll}
\Delta = \left\{x\in\Re^n: \sum^n_{i=1} x_i=1\right\}, \quad\quad & \Delta_+ = \Delta \cap \Re^n_+, \\ [5pt]
\cB^p(0;r) =\left\{x\in\Re^n: \|x\|^p_p \le r\right\}, \quad\quad & \cB^p_+(0;r) = \cB^p(0;r)\cap \Re^n_+
\ea
\]
for some $r>0$ and $p\ge 1$, and $\|x\|^p_p = \sum^n_{i=1} |x_i|^p$ for all $x\in\Re^n$.
\end{corollary}

We are ready to present some stronger results than those given in Theorem \ref{lim-point}  
under some additional assumptions. The proof of them is left in Subsection \ref{proof-main}.

\begin{theorem} \label{strong-converge}
Let $\{x^k\}$ be the sequence generated by the NPG method. Suppose that $x^*$ is an 
accumulation point of $\{x^k\}$ and  Assumption \ref{global-assump} holds. There hold:
\bi
\item[(i)]  $x^*=\proj_\cFr(x^*-t\nabla f(x^*))$ for all $t\in [0,\T]$, that is, 
$x^*$ is the unique optimal solution to the problems
\[
\min\limits_{x\in\cFr} \|x-(x^*-t\nabla f(x^*))\|^2 \quad\quad \forall t\in [0,\T];
\]
\item[(ii)] if $\|x^*\|_0<s$ and $f$ is additionally convex in $\Omega$, then 
$x^*\in \Arg\min\limits_{x\in\cFr} f(x)$, that is, $x^*$ is a global optimal solution of problem \eqref{sparse-prob};
\item[(iii)] if $\|x^*\|_0=s$ and $f$ is additionally convex in $\Omega$, then $x^*$ is a local optimal 
solution of problem \eqref{sparse-prob};
\item[(iv)] if $\|x^*\|_0=s$ and $(\cp(x^*) )_i\neq (\cp(x^*))_j$ for all $i \neq  j\in\supp(x^*)$, 
then $x^*$ is a coordinatewise stationary point. 
\ei
\end{theorem}

%
%
%

\subsection{Proof of main results} \label{proof-main}

In this subsection we present a proof for the main results, particularly, Theorems \ref{converge-fxk}, 
\ref{lim-point} and \ref{strong-converge}.  We start with the proof of Theorem \ref{converge-fxk}. 

{\bf Proof of Theorem \ref{converge-fxk}}. 
(i) Let $\cN$ be defined in \eqref{cN}. We first show that 
\beq \label{descent1}
f(x^{k+1}) \le \max\limits_{[k-M]_+ \le i \le k} f(x^i) - \frac{\sigma}{2} \|x^{k+1}-x^k\|^2 \quad\quad \forall k \in \cN
\eeq
for some $\sigma>0$. Indeed, one can observe that for every $k\in\cN$, $x^{k+1}$ is 
generated by step 2) or 3) of NPG. We now divide the proof of \eqref{descent1} 
into these separate cases. 

Case 1):  $x^{k+1}$ is generated by step 2). Then $x^{k+1}=\tx^{k+1}$ or $\hx^{k+1}$ and moreover 
$\beta(\T;x^k)\in (0,\T]$.  Using \eqref{tx} and a similar argument as for proving \eqref{fw}, one can  show that 
\beq \label{descent-tx}
f(\tx^{k+1}) \le  f(x^k) - \frac12\left(\frac{1}{\T}-L_f\right)\|\tx^{k+1}-x^k\|^2.
\eeq
Hence, if $x^{k+1}=\tx^{k+1}$, then \eqref{descent1} holds with 
$\sigma=\T^{-1}-L_f$. Moreover, such a $\sigma$ is positive due to $0<\T<1/L_f$. We next suppose 
$x^{k+1}=\hx^{k+1}$. Using this relation and the convexity of $\|\cdot\|^2$, one has
\beq \label{diffx}
\|x^{k+1}-x^k\|^2 =  \|\hx^{k+1}-x^k\|^2 \le 2 (\|\tx^{k+1}-x^k\|^2 + \|\hx^{k+1}-\tx^{k+1}\|^2).
\eeq
Summing up \eqref{descent0} and \eqref{descent-tx} and using \eqref{diffx}, we have 
\[
f(x^{k+1}) = f(\hx^{k+1}) \le f(x^k) - \frac14 \min\left\{\frac{1}{\T}-L_f,c_1\right\} \|x^{k+1}-x^k\|^2,
\]
and hence \eqref{descent1} holds with $\sigma=\min\{\T^{-1}-L_f,c_1\}/2$. 

Case 2): $x^{k+1}$ is generated by step 3).  It immediately follows from \eqref{descent} that \eqref{descent1} 
holds with $\sigma=c_2$. 

Combining the above two cases, we conclude that \eqref{descent1} holds for some $\sigma>0$. 

Let $\ell(k)$ be an integer such that $[k-M]_+ \le \ell(k) \le k$ and
\[
f(x^{\ell(k)}) = \max\limits_{[k-M]_+ \le i \le k} f(x^i).
\]
It follows from \eqref{descent1}  that $f(x^{k+1}) \le f(x^{\ell(k)})$ for every $k\in\cN$. Also, 
notice that for any $k\notin\cN$, $x^{k+1}$ must be generated by step 1b) and $f(x^{k+1})<f(x^k)$, which 
implies $f(x^{k+1}) \le f(x^{\ell(k)})$. By these facts, it is not hard to observe that 
$\{f(x^{\ell(k)})\}$ is non-increasing. In addition, recall from Assumption \ref{assump-f} that $f$ is 
bounded below in $\cC_s \cap \Omega$. Since $\{x^k\} \subseteq \cC_s \cap \Omega$, we know that $\{f(x^k)\}$ 
is bounded below and so is $f(x^{\ell(k)})$.  Hence, 
\beq \label{hatf}
\lim_{k \to \infty} f(x^{\ell(k)})=\hat f
\eeq 
for some $\hat f \in \Re$. In addition, it is not hard to observe $f(x^k) \le f(x^0)$ 
for all $k \ge 0$. Thus $\{x^k\} \subseteq \Omega_0$, where $\Omega_0$ is defined 
in Assumption \ref{assump-f}. 

We next show that 
\beq \label{fNk+1}
\lim\limits_{k \to \infty} f(x^{N(k-1)+1}) =  \hat f, 
\eeq
where $\hat f$ is given in \eqref{hatf}.  
Let 
\[
\cK_j = \{k \ge 1: \ell(k) = Nk-j \}, \quad\quad 0 \le j \le M.
\]
Notice that $Nk - M \le \ell(Nk) \le Nk$. This implies that $\{\cK_j: 0 \le j \le M\}$ forms a 
partition of all positive integers and hence $\bigcup^{M}_{j=0} \cK_j = \{1,2,\cdots\}$.
Let $0 \le j \le M$ be arbitrarily chosen such that $\cK_j$ is an infinite set. One can show that 
\beq \label{f-limit}
\lim\limits_{k\in\cK_j \to \infty} f(x^{\ell(Nk)-n_j}) = \hat f,
\eeq 
where $n_j =N-1-j$. Indeed, due to $0 \le j \le M \le N-1$, we have $n_j \ge 0$. Also, for 
every $k\in\cK_j$, we know that $\ell(Nk)-n_j = N(k-1)+1$. It then follows that 
\[
N(k-1)+1 \le \ell(Nk) - i \le Nk \quad\quad \forall 0 \le i \le n_j,  k\in\cK_j.
\] 
Thus for every $0 \le i < n_j$ and $k\in\cK_j$, we have 
$
1 \le \Mod(\ell(Nk)-i-1,N) \le N-1.
$
It follows that $x^{\ell(Nk)-i}$ must be generated by step 2) or 3) of 
NPG. This together with \eqref{descent1} implies that 
\beq \label{descent-f}
f(x^{\ell(Nk)-i}) \le f(x^{\ell(\ell(Nk)-i-1)}) - \frac{\sigma}{2} \|x^{\ell(Nk)-i}-x^{\ell(Nk)-i-1}\|^2  \quad 
\forall 0 \le i < n_j,  k\in\cK_j.
\eeq 
Letting $i=0$ in \eqref{descent-f}, one has 
\[
f(x^{\ell(Nk)}) \le f(x^{\ell(\ell(Nk)-1)}) - \frac{\sigma}{2} \|x^{\ell(Nk)}-x^{\ell(Nk)-1}\|^2  \quad 
\forall k\in\cK_j.
\]
Using this relation and \eqref{hatf}, we have 
$
\lim_{k\in\cK_j \to \infty} \|x^{\ell(Nk)}-x^{\ell(Nk)-1}\| = 0.
$
By this, \eqref{hatf}, $\{x^k\} \subseteq \Omega_0$ and the uniform continuity of $f$ in $\Omega_0$, one has 
\[
\lim\limits_{k\in\cK_j \to \infty} f(x^{\ell(Nk)-1}) = \lim\limits_{k\in\cK_j \to \infty} f(x^{\ell(Nk)}) = \hat f.
\]
Using this result and repeating the above arguments recursively for $i=1,\ldots, n_j-1$, 
we can conclude that \eqref{f-limit} holds, which, together with the fact that 
$\ell(Nk)-n_j = N(k-1)+1$ for every $k\in\cK_j$, implies that 
$
\lim_{k\in\cK_j \to \infty} f(x^{N(k-1)+1}) =  \hat f.
$ 
In view of this and $\bigcup^{M}_{j=0} \cK_j = \{1,2,\cdots\}$, one can see that \eqref{fNk+1} holds as desired.

In what follows, we show that 
\beq \label{fNk}
\lim\limits_{k \to \infty} f(x^{Nk}) =  \hat f.
\eeq
For convenience, let 
\[
\ba{l}
\cN_1 = \{k: x^{Nk+1} \ \mbox{is generated by step 2) or 3) of NPG}\}, \\ [4pt]
\cN_2 = \{k: x^{Nk+1} \ \mbox{is generated by step 1) of NPG}\}.
\ea
\]
Clearly, at least one of them is an infinite set. We first suppose that $\cN_1$ is an infinite set. 
It follows from \eqref{descent1} and the definition of $\cN_1$ that
\[
f(x^{Nk+1}) \le f(x^{\ell(Nk)}) - \frac{\sigma}{2} \|x^{Nk+1}-x^{Nk}\|^2  \quad\quad \forall k\in \cN_1,
\]  
which together with \eqref{hatf} and \eqref{fNk+1} implies 
$
\lim_{k\in\cN_1 \to \infty}\|x^{Nk+1}-x^{Nk}\| = 0.
$
Using this, \eqref{fNk+1}, $\{x^k\} \subseteq \Omega_0$ and the uniform 
continuity of $f$ in $\Omega_0$, one has 
\beq \label{fN1}
\lim\limits_{k\in\cN_1 \to \infty} f(x^{Nk}) =  \hat f.
\eeq
We now suppose that $\cN_2$ is an infinite set. By the definition of $\cN_2$, we know that 
\[
f(x^{Nk+1}) < f(x^{Nk}) \quad\quad  \forall k\in \cN_2.
\]
It then follows that 
\[
f(x^{Nk+1}) < f(x^{Nk}) \le f(x^{\ell(Nk)}) \quad\quad  \forall k\in \cN_2.
\]
This together with \eqref{hatf} and \eqref{fNk+1} leads to 
$
\lim_{k\in\cN_2 \to \infty} f(x^{Nk}) =  \hat f.
$
Combining this relation and \eqref{fN1}, one can conclude that \eqref{fNk} holds.

Finally we show that 
\beq \label{fNkj}
\lim\limits_{k \to \infty} f(x^{Nk-j}) =  \hat f \quad\quad \forall 1 \le j \le N-2.
\eeq 
One can observe that 
\[
N(k-1) +2 \le  Nk-j \le Nk-1 \quad\quad \forall 1 \le j \le N-2.
\]
Hence, $2 \le \Mod(Nk-j,N) \le N-1$ for every $1 \le j \le N-2$. It follows that 
$\{x^{Nk-j+1}\}$ is generated by step 2) or 3) of NPG for all $1 \le j \le N-2$. This together 
with \eqref{descent1} implies that  
\beq \label{descent-f1}
f(x^{Nk-j+1}) \le f(x^{\ell(Nk-j)}) - \frac{\sigma}{2} \|x^{Nk-j+1}-x^{Nk-j}\|^2  \quad\quad 
\forall 1 \le j \le N-2.
\eeq 
Letting $j=1$ and using \eqref{descent-f1}, one has 
\[
f(x^{Nk}) \le f(x^{\ell(Nk-1)}) - \frac{\sigma}{2} \|x^{Nk}-x^{Nk-1}\|^2,
\]
which together with \eqref{hatf} and \eqref{fNk} implies $\|x^{Nk}-x^{Nk-1}\| \to 0$ as $k \to \infty$. 
By this, \eqref{fNk}, $\{x^k\} \subseteq \Omega_0$ and the uniform continuity of $f$ in $\Omega_0$,  we conclude that
$
\lim_{k \to \infty} f(x^{Nk-1}) =  \hat f.
$
 Using this result and repeating the above arguments recursively for $j=2,\ldots, N-2$, we can see 
that \eqref{fNkj} holds.
 
Combining \eqref{fNk+1}, \eqref{fNk} and \eqref{fNkj}, we conclude that statement (i) of this theorem 
holds.

(ii) We now prove statement (ii) of this theorem. It follows from \eqref{descent1} that 
\[
f(x^{k+1}) \le f(x^{\ell(k)}) - \frac{\sigma}{2} \|x^{k+1}-x^k\|^2 \quad\quad \forall k\in \cN.
\] 
which together with \eqref{hatf} and statement (i) of this theorem immediately implies statement (ii) holds.
{\hfill\hbox{\vrule width1.0ex height1.0ex}\parfillskip 0pt\par} 

\gap

We next turn to prove Theorems \ref{lim-point} and \ref{strong-converge}.  Before proceeding,  
we establish several lemmas as follows.

\begin{lemma} \label{fixed-point}
Let $\{x^k\}$ be the sequence generated by the NPG method and $x^*$ an accumulation point of 
$\{x^k\}$.  There holds:
\beq \label{opt-eq1}
x^* \in \proj_\cFr \left(x^*-\hatt \nabla f(x^*)\right) 
\eeq 
for some $\hatt \in [\min\{t_{\min}, \tau/(L_f+c_2)\}, t_{\max}]$.
\end{lemma}

\begin{proof} 
 Since $x^*$ is an accumulation point of $\{x^k\}$ and $\{x^k\}\subseteq \cC_s \cap \Omega$,  one can 
observe that $x^*\in\cFr$ and moreover there exists a subsequence $\cal K$ such that 
$\{x^k\}_\cK \to x^*$. We now divide the proof of \eqref{opt-eq1} into three cases 
as follows.

Case 1): $\{x^{k+1}\}_\cK$ consists of infinite many $x^{k+1}$ that are generated by step 3) of the 
NPG method. Considering a subsequence if necessary, we assume for convenience that $\{x^{k+1}\}_\cK$ is 
generated by step 3) of NPG.  It follows from \eqref{subprob} with $\bar t_k=t_k$ that 
\[
x^{k+1} \in \Arg\min_{x \in \cFr} \left\{\left\|x-\left(x^k-t_k\nabla f(x^k)\right)\right\|^2\right\} 
\quad\quad \forall k \in\cK,
\]
which implies that for all $k\in\cK$ and $x \in \cFr$, 
\beq \label{optcond-tck}
\left\|x-\left(x^k-t_k \nabla f(x^k)\right)\right\|^2 \ge \left\|x^{k+1}-\left(x^k-t_k\nabla f(x^k)\right)\right\|^2.
\eeq
We know from \eqref{tk} that $t_k \in [\min\{t_{\min}, \tau/(L_f+c_2)\}, t_{\max}]$ for all $k\in\cK$. 
Considering a subsequence of $\cK$ if necessary, we assume without loss of generality that 
$\{t_k\}_\cK \to \hatt$ for some $\hatt \in [\min\{t_{\min}, \tau/(L_f+c_2)\}, t_{\max}]$. Notice that $\cK \subseteq \cN$, 
where $\cN$ is given in \eqref{cN}. It thus follows from Theorem \ref{converge-fxk} that $\|x^{k+1}-x^k\| \to 0$ 
as $k\in\cK \to \infty$. Using this relation, 
$\{x^k\}_\cK \to x^*$, $\{t_k\}_{\cK} \to \hatt$ and taking limits on both sides of \eqref{optcond-tck} 
as $k\in \cK \to \infty$, we obtain that 
\[
\left\|x-\left(x^*-\hatt\nabla f(x^*)\right)\right\|^2 \ge \hatt^2\|\nabla f(x^*)\|^2  \quad\quad 
\forall x \in \cFr.
\]
This together with $x^*\in\cFr$ implies that \eqref{opt-eq1} holds. 

Case 2):  $\{x^{k+1}\}_\cK$ consists of  infinite many $x^{k+1}$ that are generated by step 1) of the NPG method. 
Without loss of generality, we assume for convenience that $\{x^{k+1}\}_\cK$ is generated by step 1) of NPG. It then 
follows from NPG that $\Mod(k,N)=0$ for all $k\in\cK$. Hence, we have $\Mod(k-2,N)=N-2$ and $\Mod(k-1,N)=N-1$ for 
every $k\in\cK$, which together with $N\ge 3$ implies that  $\{x^{k-1}\}_\cK$ and $\{x^k\}_\cK$ are generated by step 2) 
or 3) of NPG. By Theorem \ref{converge-fxk}, we then have $\|x^{k-1}-x^{k-2}\| \to 0$ and $\|x^k-x^{k-1}\| \to 0$ as 
$k\in\cK \to \infty$. Using this relation and $\{x^k\}_\cK \to x^*$, we have $\{x^{k-2}\}_\cK \to x^*$ and $\{x^{k-1}\}_\cK \to x^*$. 
We now divide the rest of the proof of this case into two separate subcases as follows.

Subcase 2a): $q \neq N-1$. This together with $N\ge 3$ and $\Mod(k-1,N)=N-1$ for all $k\in\cK$ implies that $0< \Mod(k-1,N)\neq q$ 
for every $k\in\cK$. Hence, $\{x^k\}_\cK$ must be generated by step 3) of NPG.  Using this, $\{x^{k-1}\}_\cK \to x^*$ and the same 
argument as in Case 1) with $\cK$ and $k$ replaced by $\cK-1$ and $k-1$, respectively, one can conclude that \eqref{opt-eq1} holds. 

Subcase 2b): $q=N-1$. It along with $N\ge 3$ and $\Mod(k-2,N)=N-2$ for all $k\in\cK$ implies that $0<\Mod(k-2,N) \neq q$ for every $k\in\cK$. 
Thus $\{x^{k-1}\}_\cK$ must be generated by step 3) of NPG.  By this, $\{x^{k-2}\}_\cK \to x^*$ and the same argument as in Case 1) with 
$\cK$ and $k$ replaced by $\cK-2$ and $k-2$, respectively, one can see that \eqref{opt-eq1} holds. 

Case 3): $\{x^{k+1}\}_\cK$ consists of infinite many $x^{k+1}$ that are generated by step 2) of the NPG method. Without loss of generality, 
we assume for convenience that $\{x^{k+1}\}_\cK$ is generated by step 2) of NPG, which implies that $\Mod(k,N)=q$ for all $k\in\cK$. Also, 
using this and Theorem \ref{converge-fxk},  we have $\|x^{k+1}-x^k\| \to 0$ as $k\in\cK \to \infty$. This together with $\{x^k\}_\cK \to x^*$ 
yields $\{x^{k+1}\}_\cK \to x^*$. 
We now divide the proof of this case into two separate subcases as follows.

Subcase 3a): $q \neq N-1$.  It together with $\Mod(k,N)=q$ for all $k\in\cK$ 
implies that $0<\Mod(k+1,N)=q+1 \neq q$ for every $k\in\cK$. Hence, 
$\{x^{k+2}\}_\cK$ must be generated by step 3) of NPG. Using this, $\{x^{k+1}\}_\cK \to x^*$ and the same argument as in Case 1) with 
$\cK$ and $k$ replaced by $\cK+1$ and $k+1$, respectively, one can see that \eqref{opt-eq1} holds. 

Subcase 3b): $q=N-1$. This along with $N\ge 3$ and $\Mod(k,N)=q$ for all $k\in\cK$ implies that $0<\Mod(k-1,N)=q-1 \neq q$ for 
every $k\in\cK$. Thus $\{x^k\}_\cK$ must be generated by step 3) of NPG. The rest of the proof of this subcase is the same as that 
of Subcase 2a) above.
\end{proof}

\gap

\begin{lemma} \label{lem:vT}
Let $\{x^k\}$ be the sequence generated by the NPG method and $x^*$ an accumulation point of 
$\{x^k\}$. If $\|x^*\|_0=s$, then there holds:
\beq \label{vT}
\vartheta(\T;x^*)>0,
\eeq
where $\vartheta(\cdot;\cdot)$ is defined in \eqref{bT}.
\end{lemma}

\begin{proof}
 Since $x^*$ is an accumulation point of $\{x^k\}$ and $\{x^k\}\subseteq \cC_s \cap \Omega$,  one can 
observe that $x^*\in\cFr$ and moreover there exists a subsequence $\cal K$ such that  $\{x^k\}_\cK \to x^*$. Let 
\[
i(k) = \left\{\ba{ll}
\lfloor \frac{k}{N} \rfloor N +\q  & \mbox{if} \ \Mod(k,N) \neq 0, \\ [4pt]
k - N +\q  & \mbox{if} \ \Mod(k,N) = 0 
\ea\right. \quad\quad \forall k\in \cK.
\]
Clearly, $\Mod(i(k),N)=q$ and $|k-i(k)| \le N-1$ for all $k\in\tcK$. In addition, one can observe from NPG that for any 
$k\in \cK$, 
\bi
\item[] if $k< i(k)$, then $x^{k+1}, x^{k+2}, \ldots, x^{i(k)}$ are generated by step 2) or 3) of NPG;
\item[] if $k> i(k)$, then $x^{i(k)+1}, x^{i(k)+2},\ldots, x^k$ are generated by step 2) or 3) of NPG.
\ei
This, together with Theorem \ref{converge-fxk}, $\{x^k\}_\cK \to x^*$ and 
$|k-i(k)| \le N-1$ for all $k\in\cK$, implies that $\{x^{i(k)}\}_\cK \to x^*$ and $\{x^{i(k)+1}\}_\cK \to x^*$, that 
is, $\{x^k\}_\tcK \to x^*$ and $\{x^{k+1}\}_\tcK \to x^*$, where 
\[
\tcK = \{i(k):  k \in \cK\}.
\]
  In view of these, $\|x^*\|_0=s$ and $\|x^k\|_0 \le s$ for all $k$, 
 one can see that $\supp(x^k)=\supp(x^{k+1})$ for sufficiently large $k\in\tcK$. Considering 
a suitable subsequence of $\tcK$ if necessary, we assume for convenience that  
\beqa 
& \supp(x^k) =\supp(x^{k+1})  =  \supp(x^*) & \quad \forall k\in\tcK,\label{support} \\
& \|x^k\|_0=\|x^*\|_0  = s  &\quad \forall k\in\tcK. \label{card}
 \eeqa
 Also, since $\Mod(k,N)=q$ for all $k\in\tcK$, one knows that $\{x^{k+1}\}_{\tcK}$ is generated by 
 step 2) or 3) of NPG, which along with Theorem \ref{converge-fxk} implies 
 \beq \label{xdiff-tck}
 \lim\limits_{k\in\tcK \to \infty}\|x^{k+1}-x^k\| =  0  
 \eeq
We next divide the proof of \eqref{vT} into two separate cases as follows.

Case 1):  $\vartheta(\T;x^k) > \eta$ holds for infinitely many $k \in\tcK$. Then there exists a subsequence 
$\bar \cK\subseteq \tcK$ such that $\vartheta(\T;x^k) > \eta$ for all $k \in \bar \cK$. It follows from this, \eqref{gt}, 
\eqref{bT} and \eqref{support} that for all $t\in[0,\T]$ and $k\in\bar\cK$, 
\[ 
\eta < \vartheta(\T;x^k) \le  \min\limits_{i \in \supp(x^*)} \left(\cp(x^k-t\nabla f(x^k))\right)_i - 
\max\limits_{j \in [\supp(x^*)]^\c} \left(\cp(x^k-t\nabla f(x^k))\right)_j,
\] 
where $\cp$ is given in \eqref{px}. Taking the limit of this inequality as $k\in\bar\cK \to \infty$, and using $\{x^k\}_{\bar\cK} \to x^*$ 
and the continuity of $\cp$, we obtain that 
\[
\min\limits_{i \in \supp(x^*)} \left(\cp(x^*-t\nabla f(x^*))\right)_i - \max\limits_{j \in [\supp(x^*)]^\c} 
\left(\cp(x^*-t\nabla f(x^*))\right)_j \ \ge \ \eta \quad\quad \forall t \in [0,\T].
\]
This together with \eqref{gt} and \eqref{bT} yields $\vartheta(\T;x^*) \ge \eta >0$. 

Case 2): $\vartheta(\T;x^k) > \eta$ holds only for finitely many $k \in\tcK$. It implies that $\vartheta(T;x^k) \le \eta$ 
holds for infinitely many $k \in\tcK$. Considering a suitable subsequence of $\tcK$ if necessary, we 
assume for convenience that $\vartheta(\T;x^k) \le \eta$ for all $k\in\tcK$. 
This together with the fact that $\Mod(k,N)=q$ for all $k\in\tcK$ implies that $\{x^{k+1}\}_\tcK$ must be generated by step 2) if 
$\beta(\T;x^k)>0$ and by step 3) otherwise. We first show that 
\beq \label{lim-txk}
\lim\limits_{k\in\tcK \to \infty} \tx^{k+1} = x^*, 
\eeq 
where $\tx^{k+1}$ is defined in \eqref{tx}. One can observe that 
\beq \label{beta0}
\tx^{k+1}=x^k, \ f(\tx^{k+1})=f(x^k) \quad \mbox{if} \ \beta(\T;x^k)=0, k\in \tcK,
\eeq 
Also, notice that if $\beta(\T;x^k) > 0$ and $k\in \tcK$, then \eqref{descent-tx} holds. 
Combining this with \eqref{beta0}, we see that \eqref{descent-tx} holds for all $k\in \tcK$ and hence 
\beq \label{tx-x}
\|\tx^{k+1}-x^k\|^2 \le 2 (\T^{-1}-L_f)^{-1}(f(x^k)-f(\tx^{k+1})) \quad\quad \forall k\in \tcK.
\eeq 
This implies $f(\tx^{k+1}) \le f(x^k)$ for all  $k \in\tcK$. 
In addition, if $\beta(\T;x^k) > 0$ and $k\in \tcK$, one has $x^{k+1}=\tx^{k+1}$ or $\hx^{k+1}$: 
 if $x^{k+1}=\tx^{k+1}$, we have $f(x^{k+1}) =  f(\tx^{k+1})$; and  if $x^{k+1}=\hx^{k+1}$,  
 then \eqref{descent0} must hold, which yields
$
f(x^{k+1}) = f(\hx^{k+1}) \le  f(\tx^{k+1}).
$
It then follows that 
\beq \label{beta+}
f(x^{k+1}) \le  f(\tx^{k+1}) \le f(x^k) \quad \mbox{if} \ \beta(\T;x^k)>0, k\in \tcK.
\eeq 
By $\{x^k\}_\tcK \to x^*$ and \eqref{xdiff-tck}, we have $\{x^{k+1}\}_\tcK \to x^*$. Hence, 
\[
\lim\limits_{k\in\tcK \to \infty} f(x^{k+1}) = \lim\limits_{k\in\tcK \to \infty} f(x^k) = f(x^*).
\]
In view of this, \eqref{beta0} and \eqref{beta+}, one can observe that 
\[
\lim\limits_{k\in\tcK \to \infty} f(\tx^{k+1}) = \lim\limits_{k\in\tcK \to \infty} f(x^k) = f(x^*).
\]
This relation and \eqref{tx-x} lead to $\|\tx^{k+1}-x^k\| \to 0$ as $k\in\tcK \to \infty$, 
which together with $\{x^k\}_\tcK \to x^*$ implies that \eqref{lim-txk} holds as desired. 

Notice that $\|\tx^{k+1}\|_0 \le s$ for all $k\in\tcK$. 
Using this fact, \eqref{support}, \eqref{card} and \eqref{lim-txk}, one can see that there exists 
some $k_0$ such that  
\beqa 
 &\supp(\tx^{k+1})=\supp(x^k) =  \supp(x^*)  &  \quad \forall k\in\tcK, k>k_0,  \label{txk-support} \\
& \|\tx^{k+1}\|_0 = \|x^k\|_0=\|x^*\|_0  =  s & \quad \forall k\in\tcK, k>k_0. \label{txk-card}
 \eeqa
By \eqref{hx}, \eqref{txk-card}, $0<s<n$, and the definition of $\cs$, we can observe that 
$\supp(\hx^{k+1}) \neq \supp(\tx^{k+1})$ for all $k\in\tcK$ and $k>k_0$. This together with 
\eqref{support} and \eqref{txk-support} implies that 
\[
\supp(\hx^{k+1}) \neq \supp(x^{k+1}) \quad\quad \forall k\in\tcK, k>k_0,
\]
and hence $x^{k+1} \neq \hx^{k+1}$ for every $k\in\tcK$ and $k>k_0$. Using this and the fact that  
$\Mod(k,N)=\q$ and $\vartheta(\T;x^k) \le \eta$ for all $k\in\tcK$, we conclude that \eqref{descent0} must fail  
for all $k\in\tcK$ and $k>k_0$, that is, 
\beq \label{ascent}
f(\hx^{k+1}) > f(\tx^{k+1})-\frac{c_1}{2}\|\hx^{k+1}-\tx^{k+1}\|^2 \quad\quad \forall k\in\tcK, k>k_0.
\eeq 
Notice from \eqref{bT} that $\beta(\T;x^k)\in[0,\T]$. 
Considering a subsequence if necessary, we assume without loss of generality that 
\beq \label{tstar}
\lim\limits_{k\in\tcK \to \infty} \beta(\T;x^k) = t^*
\eeq
 for some $t^*\in[0,\T]$. By the definition of $\cp$, one has that for all $k\in\tcK$ and $k>k_0$,
\[
(\cp(\tx^{k+1}))_i > (\cp(\tx^{k+1}))_j \quad\quad \forall i\in\supp(\tx^{k+1}), \ j\in[\supp(\tx^{k+1})]^\c.
\]
This together with Lemma \ref{monotone-1} and \eqref{tx} implies that for all $k\in\tcK$ and $k>k_0$, 
\[
\min\limits_{i \in \supp(\tx^{k+1})} \left(\cp(x^k-\beta(\T;x^k)\nabla f(x^k))\right)_i \ \ge \  
\max\limits_{j \in[\supp(\tx^{k+1})]^\c} \left(\cp(x^k-\beta(\T;x^k)\nabla f(x^k))\right)_j. 
\]
In view of this relation and \eqref{txk-support}, one has 
\[
\min\limits_{i \in \supp(x^k)} \left(\cp(x^k-\beta(\T;x^k)\nabla f(x^k))\right)_i \ \ge \  
\max_{j \in[\supp(x^k)]^\c} \left(\cp(x^k-\beta(\T;x^k)\nabla f(x^k))\right)_j \quad \forall k\in\tcK, k>k_0,
\]
which along with \eqref{bT} implies that for all $k\in\tcK$, $k>k_0$ and $t\in[0,\T]$, 
\[
\ba{l}
\min\limits_{i \in \supp(x^k)} \left(\cp(x^k-t \nabla f(x^k))\right)_i - \max\limits_{j \in[\supp(x^k)]^\c} 
\left(\cp(x^k-t \nabla f(x^k))\right)_j  \\ [15pt] 
 \ \ \ge \min\limits_{i \in \supp(x^k)} \left(\cp(x^k-\beta(\T;x^k)\nabla f(x^k))\right)_i - 
\max\limits_{j \in[\supp(x^k)]^\c} \left(\cp(x^k-\beta(\T;x^k)\nabla f(x^k))\right)_j >0.
\ea
\]
Using this and \eqref{support}, we have that for all $k\in\tcK$ and $k>k_0$ and every $t\in [0,\T]$, 
\[
\ba{l}
\min\limits_{i \in \supp(x^*)} \left(\cp(x^k-t \nabla f(x^k))\right)_i - \max\limits_{j \in[\supp(x^*)]^\c} 
\left(\cp(x^k-t \nabla f(x^k))\right)_j  \nn \\ [15pt] 
 \ \  \ge \min\limits_{i \in \supp(x^*)} \left(\cp(x^k-\beta(\T;x^k)\nabla f(x^k))\right)_i - \max\limits_{j \in[\supp(x^*)]^\c} 
\left(\cp(x^k-\beta(\T;x^k)\nabla f(x^k))\right)_j \ > \ 0.
\ea
\]
Taking limits on both sides of this inequality as $k\in\tcK \to \infty$, and using \eqref{gt}, \eqref{tstar}, 
$\{x^k\}_\tcK \to x^*$ and the continuity of $\cp$, one can obtain 
that  $\gamma(t;x^*) \ge \gamma(t^*;x^*) \ge 0$ for every $t\in [0,\T]$. 
It then follows from this and \eqref{bT} that $\vartheta(\T;x^*) = \gamma(t^*;x^*)  \ge 0$.  

To complete the proof of \eqref{vT}, it suffices to show $\gamma(t^*;x^*) \neq 0$. Suppose on 
the contrary that 
$\gamma(t^*;x^*)  = 0$, which together with \eqref{gt} implies that 
\beq \label{amin-max}
\min\limits_{i \in \supp(x^*)}  b_i  \ = \  \max_{j \in [\supp(x^*)]^\c}  b_j,
\eeq 
where 
\beq \label{b-star}
b = \cp(x^*-t^*\nabla f(x^*)).
\eeq
Let 
\beqa
& & a^k = \tx^{k+1}-\beta(\T;x^k) \nabla f(\tx^{k+1}), \label{ak}  \\ [5pt] 
& & b^k = \cp\left(\tx^{k+1}-\beta(\T;x^k) \nabla f(\tx^{k+1})\right), \label{bk}  \\ [5pt]
& & I_k = \Arg\min\limits_{i \in \supp(\tx^{k+1})} b^k_i,  \quad
J_k = \Arg\max\limits_{j \in [\supp(\tx^{k+1})]^\c} b^k_j, \label{IJk}
\eeqa
$S_{I_k} \subseteq I_k$ and $S_{J_k} \subseteq J_k$ such that $|S_{I_k}|=|S_{J_k}|=\min\{|I_k|,|J_k|\}$. Also, let  
$S_k=\supp(\tx^{k+1})\cup S_{J_k}\setminus S_{I_k}$. 
Notice that $I_k$, $J_k$, $S_{I_k}$, $S_{J_k}$ and $S_k$ are some subsets in $\{1,\ldots,n\}$ and only 
have a finite number of possible choices. Therefore, there exists some subsequence $\hcK \subseteq \tcK$ such that 
\beq \label{const-index}
I_k = I, \ J_k = J, \ S_{I_k} = S_I, \ S_{J_k} = S_J, \ S_k = S \quad\quad \forall k\in\hcK
\eeq 
for some nonempty index sets $I$, $J$, $S_I$, $S_J$ and $S$. In view of these 
relations and \eqref{txk-support}, one can observe that 
\beqa
& I \subseteq \supp(x^*), \ J \subseteq [\supp(x^*)]^\c, \ S_I \subseteq I, \ S_J \subseteq J, &  \label{IJ}\\ [8pt] 
& |S_I|=|S_J|=\min\{|I|,|J|\}, \ S=\supp(x^*)\cup S_J \setminus S_I, & \label{SIJ}
\eeqa
and moreover, $S \neq \supp(x^*)$ and $|S|=|\supp(x^*)|=s$. In addition, by \eqref{lim-txk}, \eqref{tstar}, \eqref{ak}, \eqref{bk}, 
$\hcK \subseteq \tcK$ and the continuity of $\cp$, we see that 
\beq\label{lim-ab}
\lim\limits_{k\in\hcK \to \infty} a^k = a, \quad\quad \lim\limits_{k\in\hcK \to \infty} b^k = b,  \quad\quad b = \cp(a)
\eeq
where $b$ is defined in \eqref{b-star} and $a$ is defined as 
\beq \label{a-star}
a = x^*-t^* \nabla f(x^*).
\eeq

Claim that 
\beq \label{lim-hx}
\lim\limits_{k\in\hcK \to \infty} \hx^{k+1} = \hx^*, 
\eeq
where $\hx^*$ is defined as
\beq \label{hx-star}
\hx^*_S = \proj_{\Omega_S}(a_S), \quad \hx^*_{S^\c} = 0.
\eeq 
Indeed, by the definitions of $\cs$ and $\hx^{k+1}$, we can see that for all $k\in\hcK$, 
\[
\hx^{k+1}_{S_k} = \proj_{\Omega_{S_k}} (a^k_{S_k}), \quad \hx^{k+1}_{(S_k)^\c} = 0,
\]
which together with \eqref{const-index} yields
\[
\hx^{k+1}_S = \proj_{\Omega_{S}} (a^k_{S}) , \quad \hx^{k+1}_{S^\c} = 0 \quad \forall k \in \hcK.
\]
Using these, \eqref{lim-ab} and the continuity of $\proj_{\Omega_S}$, we immediately see that \eqref{lim-hx} and 
\eqref{hx-star} hold as desired. 

We next show that 
\beq \label{char-hxs}
\hx^* \in \Arg\min\limits_{x\in\cFr} \|x-a\|^2,
\eeq
where $a$ and $\hx^*$ are defined in \eqref{lim-ab} and \eqref{hx-star}, respectively. 
Indeed, it follows from \eqref{txk-support},  \eqref{IJk} and \eqref{const-index} that 
\beq \label{b-ineq}
\ba{lcl}
b^k_i &\le & b^k_j \quad \forall i \in I,  j\in \supp(x^*), k \in \hcK, k > k_0, \\ [5pt]
b^k_i &\ge & b^k_j \quad \forall i \in J,  j\in[\supp(x^*)]^\c, k \in \hcK, k > k_0.
\ea
\eeq
Taking limits as $k\in\hcK \to \infty$ on 
both sides of the inequalities in \eqref{b-ineq}, and using  \eqref{lim-ab}, one has
\[
\ba{lcl}
b_i &\le & b_j \quad \forall i \in I,  j\in \supp(x^*), \\ [5pt]
b_i &\ge & b_j \quad \forall i \in J,  j\in[\supp(x^*)]^\c.
\ea
\]
These together with \eqref{IJ} imply that 
\[
I \subseteq  \Arg\min\limits_{i \in \supp(x^*)} b_i, \quad\quad 
J \subseteq  \Arg\max\limits_{j \in [\supp(x^*)]^\c} b_j.
\]
By these relations, \eqref{amin-max}, \eqref{IJ} and \eqref{SIJ}, one can observe that $b_{S_I}=b_{S_J}$ 
and
\beq \label{sub-b}
b_{\supp(x^*)} = b_S,
\eeq 
where $S_I$, $S_J$ and $S$ are defined in \eqref{IJ} and \eqref{SIJ}, respectively.  
 In addition, it follows from \eqref{tx} that 
\[
\|\tx^{k+1}-(x^k-\beta(\T;x^k)\nabla f(x^k))\|^2 \le \|x-(x^k-\beta(\T;x^k)\nabla f(x^k))\|^2 \quad \forall x\in\cFr, k\in\hcK.
\]
Taking limits on both sides of this inequality as $k\in\hcK \to \infty$, and using \eqref{lim-txk}, 
\eqref{tstar}, \eqref{a-star} and $\{x^k\}_\hcK \to x^*$, one has 
\[
\|x^*-a\|^2 \le \|x-a\|^2 \quad \forall x\in\cFr,
\]  
and hence, 
\beq \label{optx}
x^* \in \Arg\min\limits_{x\in\cFr} \|x-a\|^2.
\eeq
It then follows from Lemma \ref{proj-general} that 
\beq \label{sub-x}
x^*_{\supp(x^*)} = \proj_{\Omega_{\supp(x^*)}} \left(a_{\supp(x^*)}\right).
\eeq
Recall that $|\supp(x^*)|=|S|$, which along with the symmetry of $\Omega$ implies that  
\beq \label{omega-eq}
\Omega_{\supp(x^*)}=\Omega_S.  
\eeq
We now prove that 
\beq \label{opt-xs}
\|\hx^*-a\|^2 = \|x^*-a\|^2 
\eeq
by considering two separate cases as follows.

Case i): $\Omega$ is nonnegative symmetric.  This together with  $b=\cp(a)$ and \eqref{px} yields 
$b=a$. Using this and \eqref{sub-b},  we can observe that 
\beq \label{sub-a}
a_S = a_{\supp(x^*)}, \quad \quad a_{S^\c}  = a_{[\supp(x^*)]^\c}.
\eeq 
In view of these, \eqref{hx-star}, \eqref{sub-x} and \eqref{omega-eq}, one has 
$\hx^*_S = x^*_{\supp(x^*)}$. Using this, \eqref{hx-star} and \eqref{sub-a}, we have
\[
\|\hx^*-a\|^2 = \|\hx^*_S-a_S\|^2 + \|a_{S^\c}\|^2 = \|x^*_{\supp(x^*)}-a_{\supp(x^*)}\|^2 + \left\|a_{[\supp(x^*)]^\c}\right\|^2 = \|x^*-a\|^2.
\]

Case ii): $\Omega$ is sign-free symmetric. It implies that $\Omega_S$ is also  
sign-free symmetric. Using this fact, $|S|=s$,  Lemma \ref{relation-projs}, \eqref{hx-star}, \eqref{sub-x}, 
\eqref{omega-eq} and \eqref{sub-a}, we obtain that 
\beqa
\hx^*_S &=& \sign(a_S) \circ \proj_{\Omega_S \cap \Re^s_+}(|a_S|),  \label{hxs} \\
x^*_{\supp(x^*)} &=& \sign(a_{\supp(x^*)}) \circ \proj_{\Omega_S \cap \Re^s_+}(|a_{\supp(x^*)}|). \label{xs1}
\eeqa
Notice that $b=\cp(a)$. Using  this relation, \eqref{px}, \eqref{sub-b} and \eqref{sub-a}, one can have  
\beq \label{aS}
|a_S| = |b_S| = |b_{\supp(x^*)}| = \left|a_{\supp(x^*)}\right|.
\eeq  
Using \eqref{hxs}, \eqref{xs1} and \eqref{aS}, we can observe that 
\[
|\hx^*_S| = \left|x^*_{\supp(x^*)}\right|, \quad \quad a_S \circ \hx^*_S = a_{\supp(x^*)} \circ x^*_{\supp(x^*)}. 
\]
In view of these two relations and \eqref{aS}, one can obtain that   
\[
\ba{lcl}
\|\hx^*-a\|^2 &=& \|\hx^*_S-a_S\|^2 + \|a_{S^\c}\|^2 = \|\hx^*_S\|^2 - 2 \left(a_S\right)^T \hx^*_S + \|a\|^2, \\ [8pt]
&=& \left\|x^*_{\supp(x^*)}\right\|^2 - 2  \left(a_{\supp(x^*)}\right)^T x^*_{\supp(x^*)} + \|a\|^2  \ = \ \|x^*-a\|^2. 
\ea
\]
Combining the above two cases, we conclude that \eqref{opt-xs} holds. In addition, notice from 
\eqref{hx-star} and $|S|=s$ that $\hx^*\in\cFr$. In view of this, \eqref{optx} and \eqref{opt-xs}, we conclude 
that \eqref{char-hxs} holds as desired. 

Recall from above that $|\hx^*_S| = |x^*_{\supp(x^*)}|$, which together with $\|x^*\|_0=s$ 
and $\|\hx^*\|_0 \le s$ implies $\supp(\hx^*)=S$. Notice that  $S \neq \supp(x^*)$. It then follows $\hx^* \neq x^*$. Using this, 
\eqref{a-star} and \eqref{char-hxs}, one observe that $t^* \neq 0$ and hence $t^* \in (0, \T]$. 
By this relation, \eqref{a-star}, 
\eqref{char-hxs}, and a similar argument as for proving \eqref{fw}, we can obtain that 
\[
f(\hx^*) \le f(x^*) - \frac12\left(\frac{1}{\T}-L_f\right)\|\hx^*-x^*\|^2. 
\]
Using this relation,  \eqref{lim-txk},\eqref{lim-hx}, $c_1\in(0,1/\T-L_f)$, $\hcK\subseteq\tcK$ and $\hx^* \neq x^*$, 
one can observe that for all sufficiently large $k\in\hcK$, 
\[
f(\hx^{k+1}) \le f(\tx^{k+1}) - \frac{c_1}{2}\|\hx^{k+1}-\tx^{k+1}\|^2,
\]
which together with $\hcK\subseteq\tcK$ yields a contradiction to \eqref{ascent}. Therefore, \eqref{vT} holds 
and this completes the proof.
\end{proof}

\gap

%
%

\begin{lemma} \label{soln-path}
Suppose that $x^*\in\Re^n$ satisfies $\|x^*\|_0=s$ and  
\beq \label{soln-t1}
x^* \in \proj_{\cFr} (x^*-t_1\nabla f(x^*))
\eeq
for some $t_1>0$. In addition, assume $\gamma(t_2;x^*) \ge 0$ for some $t_2>0$, 
where $\gamma(\cdot;\cdot)$ is defined in \eqref{gt}. Then there holds
\[
x^* \in \proj_\cFr(x^*-t_2\nabla f(x^*)).
\]
\end{lemma}

\begin{proof}
For convenience, let 
\[
a=x^*-t_1\nabla f(x^*), \quad  b=x^*-t_2\nabla f(x^*), \quad T=\supp(x^*).
\] 
 In view of these, \eqref{gt} and $\gamma(t_2;x^*) \ge 0$, one has 
$
\min\limits_{i \in T} \left(\cp(b)\right)_i  \ \ge \ \max\limits_{j \in T^\c} \left(\cp(b)\right)_j, 
$
which together with $\|x^*\|_0=s$ implies there exists some $\sigma \in \tsigma(\cp(b))$  
such that $\cS^\sigma_{[1,s]} = T$. Let $y\in\Re^n$ be given as follows:
\[
y_T = \proj_{\Omega_T}(b_T), \quad\quad y_{T^\c} = 0.
\]
It then follows from Theorem \ref{opt-proj} that $y\in\proj_{\cC_s \cap \Omega}(b)$. To complete 
this proof, it suffices to show $y_T=x^*_T$. Indeed, by $T=\supp(x^*)$, \eqref{soln-t1} and Lemma  
\ref{proj-general}, one has $x^*_T = \proj_{\Omega_T}(a_T)$, which together with the convexity of 
$\Omega_T$ yields
\[
x^*_T = \arg\min\{\|z-a_T\|^2: z \in \Omega_T\}.
\]
By the first-order optimality condition of this problem and the definition of $a$, one has 
$
-t_1 [\nabla f(x^*)]_T \in \cN_{\Omega_T}(x^*_T). 
$
Since $t_1,t_2>0$, we immediately have 
$
-t_2 [\nabla f(x^*)]_T \in \cN_{\Omega_T}(x^*_T),
$ 
which along with the definition of $b$ and the convexity of 
$\Omega_T$ implies $x^*_T=\proj_{\Omega_T}(b_T)$. Hence, $y_T=x^*_T$ as desired.
\end{proof}

\gap

\begin{lemma}\label{loc-suff}
Suppose that $f$ is additionally convex, $x^*\in\Re^n$ satisfies $\|x^*\|_0=s$, and moreover 
\beq \label{opt-Arg}
x^* \in \Arg\min\limits_{x\in \cFr} \|x-(x^*-t\nabla f(x^*))\|^2
\eeq
for some $t>0$. Then $x^*$ is a local optimal solution of problem \eqref{sparse-prob}. 
\end{lemma}

\begin{proof}
Let $J=[\supp(x^*)]^\c$ and $\tomega = \{x\in\Omega:  x_J = 0\}$. It is not hard to 
observe from \eqref{opt-Arg} that 
\[
x^* = \arg\min\limits_{x\in \tomega} \|x-(x^*-t\nabla f(x^*))\|^2,
\]
whose first-order optimality condition leads to $-\nabla f(x^*) \in \cN_{\tomega}(x^*)$. 
This together with convexity of $\tomega$ and $f$ implies that
$
x^* \in \Arg\min\{f(x): x\in\tomega\}.
$
It then follows that $f(x) \ge f(x^*)$ for all $x\in  \widetilde \cO(x^*;\epsilon)$, where  
$ \widetilde\cO(x^*;\epsilon) = \{x\in\tomega: \|x-x^*\| < \epsilon\}$ and  $\epsilon = \min\{|x^*_i|: i \in J^\c\}$.  
 By the definition of $\epsilon$ and $\|x^*\|_0=s$, one can observe that 
$\cO(x^*;\epsilon)=\widetilde \cO(x^*;\epsilon)$, where $\cO(x^*;\epsilon) = \{x\in\cFr: \|x-x^*\| < \epsilon\}$. 
We then conclude that $f(x) \ge f(x^*)$ for all $x\in \cO(x^*;\epsilon)$, which implies 
that $x^*$ is a local optimal solution of problem \eqref{sparse-prob}.
\end{proof}


\begin{lemma} \label{opt-all-t}
Suppose that $x^*\in\Re^n$ satisfies $\|x^*\|_0<s$ and moreover 
\beq \label{opt-Arg2}
x^* \in \proj_\cFr(x^*-\hatt\nabla f(x^*)).
\eeq
for some $\hatt>0$. Under Assumption \ref{global-assump}, there hold:
\bi
\item[(i)] $x^*=\proj_\cFr(x^*-t\nabla f(x^*))$ for all $t \ge 0$, that is, $x^*$ is the unique 
optimal solution to the problem 
\[
\min\limits_{x\in\cFr} \|x-(x^*-t\nabla f(x^*))\|^2.
\]
\item[(ii)] if $f$ is additionally convex in $\Omega$, then  $x^*\in \Arg\min\limits_{x\in \cFr} f(x)$, 
that is, $x^*$ is an optimal solution of $f$ over $\cFr$.
\ei
\end{lemma}

\begin{proof}
For convenience, let $a = x^*-\hatt \nabla f(x^*)$. It follows from \eqref{opt-Arg2} and Lemma 
\ref{proj-general} that 
\beq \label{projT}
x^*_T = \proj_{\Omega_T}(a_T) \quad\quad \forall T \in \cT_s(x^*).
\eeq 
Using this, $\hatt >0$, the definition of $a$, and the first-order optimality condition of the 
associated  optimization problem with \eqref{projT}, we have 
$-[\nabla f(x^*)]_T \in \cN_{\Omega_T}(x^*_T)$ for every $T \in \cT_s(x^*)$.
This together with Assumption \ref{global-assump} yields
\beq \label{1st-opt}
-\nabla f(x^*) \in \cN_{\Omega}(x^*).
\eeq
Using this and the convexity of $\Omega$ and $f$, we have $
x^* = \proj_\Omega(x^*-t\nabla f(x^*))$  for all $t \ge 0$.
It then follows that statement (i) holds due to $x^*\in\cFr \subseteq \Omega$.  In addition, 
by \eqref{1st-opt} and the convexity of $f$ and $\Omega$, one can see that 
$x^*\in \Arg\min\{f(x): x\in\Omega\}$, which together with $x^*\in\cFr \subseteq \Omega$ implies that 
 statement (ii) holds. 
\end{proof}

\gap

We are now ready to prove Theorem \ref{lim-point}. 


{\bf Proof of Theorem \ref{lim-point}}. 
Let $\{x^k\}$ be the sequence generated by the NPG method and $x^*$ an accumulation point of $\{x^k\}$.
In view of Lemma \ref{fixed-point}, there exists $\hatt \in [\min\{t_{\min}, \tau/(L_f+c_2)\}, t_{\max}]$ 
such that \eqref{opt-eq1} holds. 

(i) It follows from \eqref{opt-eq1} and \cite[Theorem 5.2]{BeHa14} that 
$x^*\in \proj_\cFr(x^*-t \nabla f(x^*))$ for every $t \in[0,\hatt]$.
Using this relation, $\|x^*\|_0<s$ and Theorem \ref{unique-soln}, we know that 
$x^*=\proj_\cFr(x^*-t\nabla f(x^*))$ for all $t\in [0,\hatt]$. Hence, statement (i) of this theorem holds.

(ii) Suppose  $\|x^*\|_0=s$. In view of Lemma \ref{lem:vT}, we know that $\vartheta(\T;x^*)>0$, which 
together with \eqref{bT} implies that $\gamma(t;x^*)>0$ for all $t\in[0,\T]$. Using this, \eqref{opt-eq1} and 
Lemma \ref{soln-path}, we obtain that 
\[
x^*\in \proj_\cFr(x^*-t \nabla f(x^*)) \quad\quad \forall t \in[0,\T].
\]
Using this, \eqref{gt}, Theorem \ref{unique-soln1} and the fact that 
$\gamma(t;x^*)>0$ for all $t\in[0,\T]$, we further have
\[
x^*=\proj_\cFr(x^*-t \nabla f(x^*)) \quad\quad \forall t \in[0,\T].
\]

(iii) Let $\hatt \in [\min\{t_{\min}, \tau/(L_f+c_2)\}, t_{\max}]$ be given in statement (i) of  this 
theorem. In view of \eqref{special-param}, one can observe $\hatt \ge \T$. The conclusion of this 
statement then immediately follows from  statements (i) and (ii) of this theorem.

(iv) Suppose that $\|x^*\|_0=s$ and $f$ is convex. It follows from \eqref{opt-eq1} and Lemma 
\ref{loc-suff} that $x^*$ is a local optimal solution of problem \eqref{sparse-prob}.  

(v) Suppose that $\|x^*\|_0=s$ and $(\cp(x^*) )_{\bar i}\neq (\cp(x^*))_{\bar j}$ for all 
$\bar i\neq \bar j\in\supp(x^*)$. 
We will show that $x^*$ is a coordinatewise stationary point.  Since $x^*$ is an accumulation point of 
$\{x^k\}$, there exists a subsequence $\cK$ such that $\{x^k\}_\cK \to x^*$. For every $k\in \cK$, 
let  $\iota(k)$ be the unique integer in $[k, k+N)$ such that $\Mod(\iota(k),N)=0$, and let
\[
\tcK = \{\iota(k): k \in \cK\}, \quad\quad \hcK = \{k\in\tcK: x^{k+1} \ \mbox{is generated by step 1)} \}.
\]
In addition, for each $k\in\tcK$, let $i_k,j_k$ be chosen by the subroutine $\cswap$, which satisfy
\beqa 
&& i_k \in \Arg\min\{(\cp(-\nabla f(x^k)))_\ell: \ell \in I_k\},  \label{ik}\\ [5pt]
&& j_k \in \Arg\max\{(\cp(-\nabla f(x^k)))_\ell: \ell \in [\supp(x^k)]^\c\}, \label{jk}
\eeqa
where 
\beq \label{Ik}
I_k= \Arg\min\{(\cp(x^k))_i: i \in \supp(x^k)\}.
\eeq
 It is not hard observe that if $\iota(k) \neq k$  for some $k\in\cK$, then $x^{k+1}, \ldots, 
x^{\iota(k)}$ are generated by step 2) or 3) of NPG. Using this observation, Theorem 
\ref{converge-fxk} and 
$0 \le \iota(k)-k<N$ for all $k\in\cK$, one can see that $\{x^k\}_{\tcK} \to x^*$. By this and $\|x^*\|_0=s$, 
there exists $k_0$ such that $\supp(x^k)=\supp(x^*)$ for all $k\in\tcK$ and $k>k_0$. Also, notice 
that there are only finite number of possible choices for $I_k$, $i_k$ and $j_k$. Considering a subsequence of $\hcK$ 
if necessary,  we assume for convenience that for all $k\in\hcK$, 
$I_k \equiv I$, $i_k \equiv i$, $j_k \equiv j$ for some $I$, $i$ and $j$. In view of these, \eqref{ik}, \eqref{jk}, 
\eqref{Ik}, $\{x^k\}_\tcK \to x^*$, and the continuity of $\cp$ and $\nabla f$, one can obtain that 
\beqa
&& i \in \Arg\min\{(\cp(-\nabla f(x^*)))_\ell: \ell \in I\},  \label{cw-cond1} \\ [5pt]
&& j \in \Arg\max\{(\cp(-\nabla f(x^*)))_\ell: \ell \in [\supp(x^*)]^\c\}, \label{cw-cond2}\\ [5pt]
&& I \subseteq \Arg\min\{(\cp(x^*))_\ell: \ell \in \supp(x^*)\}. \nn
\eeqa
The last relation along with the assumption $(\cp(x^*))_{\bar i} \neq (\cp(x^*))_{\bar j}$ for all $\bar i\neq \bar j \in \supp(x^*)$ 
implies that 
\beq \label{cw-cond3}
 I = \Arg\min\{(\cp(x^*))_\ell: \ell \in \supp(x^*)\}. 
\eeq
We next show that 
\beq \label{cw-cond4}
f(x^*) \le \left\{\ba{lr} 
\min\{f(x^*-x^*_i \be_i+x^*_i \be_j), f(x^*-x^*_i \be_i-x^*_i \be_j)\} & \mbox{if} \ \Omega \ 
\mbox{is sign-free} \\ 
&\mbox{symmetric};  \\ [8pt]
f(x^*-x^*_i \be_i+x^*_i \be_j) &\mbox{if} \ \Omega \ 
\mbox{is  nonnegative} \\ 
&\mbox{symmetric}.
\ea\right.
\eeq
by considering two separate cases as follows.

Case 1): $\hcK$ is an infinite set. By Theorem \ref{converge-fxk}, $\{x^k\}_\tcK \to x^*$ and 
the continuity of $f$,  we have 
\beq \label{fxk+1}
\lim\limits_{k\in\hcK \to \infty} f(x^{k+1}) = \lim\limits_{k\in\hcK \to \infty} f(x^k) = f(x^*).
\eeq 
 In addition, by the definitions of $\hat \cK$ and $\cswap$, one can observe that for each $k\in\hcK$, 
\[
f(x^{k+1}) \le \left\{\ba{lr} 
\min\{f(x^k-x^k_{i_k} \be_i+x^k_{i_k} \be_{j_k}), f(x^k-x^k_{i_k} \be_{i_k}-x^k_{i_k} \be_{j_k})\} & \mbox{if} \ \Omega \ 
\mbox{is sign-free} \\ 
&\mbox{symmetric};  \\ [8pt]
f(x^k-x^k_{i_k} \be_i+x^k_{i_k} \be_{j_k}) & \mbox{if} \ \Omega \ 
\mbox{is  nonnegative} \\ 
&\mbox{symmetric}.
\ea\right.
\]
Taking limits as $k\in\hcK \to \infty$ on both sides of this relation and using \eqref{fxk+1}, we see 
that \eqref{cw-cond4} holds.  

Case 2): $\hcK$ is a finite set. It follows that $\bar\cK=\tcK\setminus\hcK$ is an infinite set. 
Moreover, by the definitions of $\bar\cK$ and $\cswap$, one can observe that for every $k\in\bar\cK$, 
\[
f(x^{k}) \le \left\{\ba{lr} 
\min\{f(x^k-x^k_{i_k} \be_i+x^k_{i_k} \be_{j_k}), f(x^k-x^k_{i_k} \be_{i_k}-x^k_{i_k} \be_{j_k})\} & \mbox{if} \ \Omega \ 
\mbox{is sign-free} \\ 
&\mbox{symmetric};  \\ [8pt]
f(x^k-x^k_{i_k} \be_i+x^k_{i_k} \be_{j_k}) & \mbox{if} \ \Omega \ 
\mbox{is  nonnegative} \\ 
&\mbox{symmetric}.
\ea\right.
\]
Taking limits as $k\in\bar\cK \to \infty$ on both sides of this relation, and using $\{x^k\}_\tcK \to x^*$ and the 
continuity of $f$, we conclude that \eqref{cw-cond4} holds.  

In addition, by Lemmas \ref{proj-general} and \ref{fixed-point}, we know that 
\[
x^*_T = \proj_{\Omega_T}(x^*_T-\hatt(\nabla f(x^*))_T) \quad\quad \forall T\in\cT_s(x^*)
\]
for some $\hatt \in [\min\{t_{\min}, \tau/(L_f+c_2)\}, t_{\max}]$. Combining this relation with 
\eqref{cw-cond1}-\eqref{cw-cond4}, we see that $x^*$ is a coordinatewise stationary point. 
{\hfill\hbox{\vrule width1.0ex height1.0ex}\parfillskip 0pt\par}

\gap 

Finally we present a proof for Theorem \ref{strong-converge}.
  
{\bf Proof of Theorem \ref{strong-converge}}.  
Let $\{x^k\}$ be the sequence generated by the NPG method and $x^*$ an 
accumulation point of $\{x^k\}$. We divide the proof of statement (i) of this theorem 
into two separate cases. 

Case 1): $\|x^*\|_0=s$. It then follows from Theorem \ref{lim-point} that statement (i) of 
Theorem \ref{strong-converge}  holds.

Case 2): $\|x^*\|_0<s$. By Theorem \ref{lim-point}, there exists some $\hatt>0$ such that 
\beq \label{opt-hatt}
x^* \in \proj_\cFr(x^*-\hatt\nabla f(x^*)).
\eeq 
In view of this relation, Assumption \ref{global-assump} and Lemma \ref{opt-all-t}, we again 
see that statement (i) of this theorem holds. 

 Statement (ii) of this theorem immediately follows from \eqref{opt-hatt}, Assumption  
\ref{global-assump} and Lemma \ref{opt-all-t}.  In addition, statements (iii) and (iv) of 
this theorem hold due to statements (iii) and (iv) of Theorem \ref{lim-point}. 
{\hfill\hbox{\vrule width1.0ex height1.0ex}\parfillskip 0pt\par}

\section{Numerical results}
\label{results}

In this section we conduct numerical experiment to compare the performance of our NPG method and 
the PG method with a constant stepsize.  In particular, we apply these methods to problem 
\eqref{sparse-prob} with $f$ being chosen as a least squares or a logistic loss. All codes are 
written in MATLAB and all computations are performed on a MacBook Pro running with Mac OS X Lion 10.7.4 and 4GB memory.

Recall that the PG method with a constant stepsize $\alpha$ generates the iterates according to
\[
x^{k+1} = \proj_{\cFr}(x^k-\alpha\nabla f(x^k)) \quad\quad \forall k \ge 0
\] 
for some $\alpha\in (0, 1/L_f)$, where $L_f$ is the Lipschitz constant of $\nabla f$. In our experiments, 
we set $\alpha=0.995/L_f$.  For our NPG method, we set $\T=0.995/L_f$, 
$t_{\min}=\T$, $t_{\max} = 10^8$, $c_1=\min(0.995(1/\T-L_f),10^{-8})$, $c_2=10^{-4}$, $\tau=2$, 
$\eta=10^3$. In addition, we set $t^0_0=1$, and  update $t^0_k$ by the 
same strategy as used in \cite{BaBo88,BiMaRa00,WrNoFi09},
that is,
\[
t^0_k = \max\left\{t_{\min},\min\left\{t_{\max},\frac{\|\Delta x\|^2}{|\Delta x^T \Delta g|}\right\}\right\},
\]
where $\Delta x = x^k -x^{k-1}$ and $\Delta g = \nabla f(x^k) - \nabla f(x^{k-1})$. Both methods 
terminate according to the criterion $|f(x^k)-f(x^{k-1})| \le 10^{-8}$.

In the first experiment we compare the performance of NPG and PG for solving problem 
\eqref{sparse-prob} with $\Omega=\Re^n$ and 
\[
f(x)=\frac12 \|Ax-b\|^2  \quad\quad\quad\quad\quad\quad  \mbox{(least squares loss)}.
\]
The matrix $A \in \Re^{m \times n}$ and the vector $b\in \Re^m$ are randomly 
generated in the same manner as described in $l_1$-magic \cite{CanRom05}. In particular, 
given $\sigma >0$ and positive integers $m$, $n$, $s$ with $m < n$ and $s< n$, 
we first generate a matrix $W\in\Re^{n\times m}$ with entries randomly chosen 
from a standard normal distribution. We then compute an orthonormal basis, denoted by $B$, for the range space 
of $W$, and set $A=B^T$. In addition, we randomly generate a vector $\tx \in \Re^n$ with 
only $s$ nonzero components that are $\pm 1$, and generate a vector $v\in\Re^{m}$ 
with entries randomly chosen from a standard normal distribution. Finally, we set $b = A\tx+\sigma v$. In 
particular, we choose $\sigma=0.1$ for all instances. 

We choose $x^0=0$ as the initial point for both methods, and set $M=4$, $N=5$, $\q=3$ for NPG. 
The computational results are presented in Table \ref{res-l1}. In detail, the parameters 
$m$, $n$ and $s$ of each instance are listed in the first three columns, respectively. The 
cardinality of the approximate solution found by each method is presented in next two columns.  
The objective function value of \eqref{sparse-prob}  for these
methods is given in columns six and seven, and CPU times (in seconds) are given in the last two columns,
respectively. One can observe that both methods are comparable in terms of CPU time, but NPG 
substantially outperforms PG in terms of objective values.

\begin{table}[t!]
\caption{PG and NPG methods for least squares loss}
\centering
\label{res-l1}
\begin{tabular}{|rrr||cc||cc||rr|}
\hline
\multicolumn{3}{|c||}{Problem} &  \multicolumn{2}{c||}{Solution Cardinality} & 
\multicolumn{2}{c||}{Objective Value} &
 \multicolumn{2}{c|}{CPU Time} \\
\multicolumn{1}{|c}{m} & \multicolumn{1}{c}{n} &  \multicolumn{1}{c||}{s}
& \multicolumn{1}{c}{\sc PG} & \multicolumn{1}{c||}{\sc NPG} 
& \multicolumn{1}{c}{\sc PG} & \multicolumn{1}{c||}{\sc NPG} 
& \multicolumn{1}{c}{\sc PG} &  \multicolumn{1}{c|}{\sc NPG} \\
\hline
120 & 512 & 20 &  20 & 20 & 0.61 & 0.38 & 0.02 & 0.05 \\
240 & 1024 & 40 & 40 & 40 & 1.30 & 0.87 & 0.03 & 0.06 \\
360 & 1536 & 60 & 60 & 60 & 2.42 & 1.44 & 0.04 & 0.08 \\
480 & 2048 & 80 & 80 & 80 & 2.57 & 1.86 & 0.09 & 0.10 \\
600 & 2560 & 100 & 100 & 100 & 3.46 & 2.36 & 0.19 & 0.18\\
720 & 3072 & 120 & 120 & 120 & 4.21 & 2.77 & 0.34 & 0.31 \\
840 & 3584 & 140& 140 & 140 & 5.42 & 3.44 & 0.49 & 0.37 \\
960 & 4096 & 160& 160 & 160 & 5.76 & 3.92 & 0.64 & 0.46 \\
1080 & 4608 & 180& 180 & 180 & 6.85 & 3.94 & 0.55 & 0.76 \\
1200 & 5120 & 200& 200 & 200 & 8.07 & 4.75 & 0.95 & 0.84 \\
\hline
\end{tabular}
\end{table}

In the second experiment, we compare the performance of NPG and PG for solving problem \eqref{sparse-prob} with $\Omega=\Re^n$, $s=0.01n$ and 
\beq \label{lsq}
f(x)=\sum^m_{i=1} \log(1+\exp(-b_i (a^i)^Tx))    \quad\quad\quad\quad\quad\quad  \mbox{(logistic loss)}.
\eeq 
It can be verified that the Lipschiz constant of $\nabla f$ is $L_f = \|\tilde A\|^2$,
where
$
\tilde A = \left[
b_1 a^1, \cdots, b_m a^m \right].
$

The samples $\{a^1, \ldots, a^m\}$ and the corresponding outcomes $b_1, \ldots, b_m$ are generated 
in the same manner as described in \cite{KoKiBo07}. In detail, for each instance we choose equal 
number of positive and negative samples, that is, $m_+  =m_- = m/2$, where $m_+$ (resp., $m_-$) 
is the number of samples with outcome $+1$ (resp., $-1$). The features of positive (resp., negative) 
samples are independent and identically distributed, drawn from a normal distribution $N(\mu,1)$, 
where $\mu$ is in turn drawn from a uniform distribution on $[0,1]$ (resp., $[-1,0]$). 

We choose $x^0=0$ as the initial point for both methods, and set $M=2$, $N=3$, $\q=2$ for NPG. The results of NPG and PG 
for the instances generated above are presented in Table \ref{res-logistic}.  We observe that NPG outperforms 
 PG in terms of objective value and moreover it is substantially superior to PG in CPU time.

\begin{table}[t!]
\caption{PG and NPG methods for logistic loss}
\centering
\label{res-logistic}
\begin{tabular}{|rr||cc||cc||rr|}
\hline
\multicolumn{2}{|c||}{Problem} &  \multicolumn{2}{c||}{Solution Cardinality} & 
\multicolumn{2}{c||}{Objective Value} &
 \multicolumn{2}{c|}{CPU Time} \\
\multicolumn{1}{|c}{m} & \multicolumn{1}{c||}{n} 
& \multicolumn{1}{c}{\sc PG} & \multicolumn{1}{c||}{\sc NPG} 
& \multicolumn{1}{c}{\sc PG} & \multicolumn{1}{c||}{\sc NPG} 
& \multicolumn{1}{c}{\sc PG} &  \multicolumn{1}{c|}{\sc NPG} \\
\hline
500 & 1000 & 10 & 10 & 304.0& 301.4 & 6.0 & 0.3 \\ 
1000 & 2000 & 20 & 20 & 616.4 & 606.9 & 75.2 & 0.3 \\
1500 & 3000 & 30 & 30 & 978.1 & 912.4 &  263.4 & 1.1 \\
2000 & 4000 & 40 & 40 & 1286.8 & 1215.6 & 425.3 & 1.8 \\
2500 & 5000 & 50 & 50 & 1588.3 & 1522.0 & 972.3 & 2.7 \\
3000 & 6000 & 60 & 60 & 1819.1 & 1861.3 & 1560.5 & 5.6 \\
3500 & 7000 & 70 & 70 & 2241.2 & 2129.7 & 2321.3 & 5.0\\
4000 & 8000 & 80 & 80 & 2514.3 & 2417.8 & 3699.1 & 10.5 \\
4500 & 9000& 90 & 90 & 2760.6 & 2725.8 & 5568.9 & 11.4 \\
5000 & 10000 & 100 & 100 & 3284.9 &3008.6 & 7813.6 & 12.5 \\
\hline
\end{tabular}
\end{table}

\begin{table}[t!]
\caption{PG and NPG methods for least squares loss over sparse simplex}
\centering
\label{res-simplex}
\begin{tabular}{|rr||cc||cc||rr|}
\hline
\multicolumn{2}{|c||}{Problem} &  \multicolumn{2}{c||}{Solution Cardinality} & 
\multicolumn{2}{c||}{Objective Value} &
 \multicolumn{2}{c|}{CPU Time} \\
\multicolumn{1}{|c}{m} & \multicolumn{1}{c||}{n} 
& \multicolumn{1}{c}{\sc PG} & \multicolumn{1}{c||}{\sc NPG} 
& \multicolumn{1}{c}{\sc PG} & \multicolumn{1}{c||}{\sc NPG} 
& \multicolumn{1}{c}{\sc PG} &  \multicolumn{1}{c|}{\sc NPG} \\
\hline
100 & 500 & 5 & 5 & 202.8 & 108.3 & 0.06 & 0.07 \\
200 & 1000 & 10 & 10 & 400.6 & 151.4 & 0.08 & 0.10 \\
300 & 1500 & 15 & 15 & 556.9 & 226.6 & 0.10 & 0.13 \\
400 & 2000 & 20 & 20 & 774.0 & 336.2 & 0.25 & 0.27 \\
500 & 2500 & 25 & 25 & 1020.2 & 382.8 & 0.36 & 0.44 \\
600 & 3000 & 30 & 30 & 1175.4 & 426.9 & 0.48 & 0.77 \\
700 & 3500 & 35 & 35 & 1311.6 & 534.0 & 0.59 & 0.81 \\
800 & 4000 & 40 & 40 & 1535.3 & 587.0 & 0.86 & 1.52 \\
900 & 4500 & 45 & 45 & 1777.2 & 670.6 & 1.21 & 1.76 \\
1000 & 5000 & 50 & 50 &1961.5 & 772.0 & 1.25 & 2.24 \\
\hline
\end{tabular}
\end{table}

In the last experiment we compare the performance of NPG and PG for solving problem \eqref{sparse-prob} with a least squares loss $f$ defined in \eqref{lsq}, $s=0.01n$, and $\Omega=\Delta_+$, where $\Delta_+$ is the $n$-dimensional 
nonnegative simplex defined in Corollary \ref{special-sets}.  The 
associated problem data $A$ and $b$ are randomly generated as follows. We first randomly generate an orthonormal matrix $\bar A$ in the same manner as described in the first experiment above. Then 
we obtain $A$ by pre-multiplying $\bar A$ by the diagonal matrix $D$ whose $i$th diagonal entry is $i^2$ for $i=1,\ldots, n$.  In addition, we generate a  vector $z\in \Re^n$ whose 
entries are randomly chosen according to the uniform distribution in $[0,1]$,  and set $b=Az/\|z\|_1$.  

We choose $x^0=(\sum^s_{i=1} \be_i)/s$ as an initial point for both methods, and set $M=3$, $N=4$, $\q=3$ for NPG. The results of NPG and PG for those instances are presented in Table \ref{res-simplex}.  We observe that NPG is comparable to PG  in terms of CPU time, but it is significantly superior to PG in objective value.

\section{Concluding remarks}
\label{conclude}

In this paper we considered the problem of minimizing a Lipschitz differentiable function over a class of sparse 
symmetric sets. In particular we introduced a new optimality condition that is proved to be stronger than the $L$-stationarity optimality condition introduced in \cite{BeHa14}. We also proposed a nonmonotone projected gradient 
(NPG) method for solving this problem by  incorporating some support-changing and coordintate-swapping 
strategies into a projected gradient with variable stepsizes. It was shown that any accumulation point of NPG 
satisfies the new optimality condition. The classical projected gradient (PG) method with a constant stepsize, however, generally does not possess such a property. 

It is not hard to observe that a similar optimality condition as the one stated in Theorem \ref{opt-cond-prop1} 
can be derived for the problem 
\beq \label{general-prob}
\min\{f(x): x \in \cX\},
\eeq
where $\cX$ is closed but possibly nonconvex and $f$ satisfies \eqref{lipschitz}. That is, for any optimal solution  $x^*$ of \eqref{general-prob}, there holds
\[
x^*= \proj_{\cX}(x^*-t\nabla f(x^*)) \quad\quad \forall t\in[0,1/L_f).
\]
It can be easily shown that any accumulation point $x^*$ of the sequence generated by the classical 
PG method with a constant stepsize $\T \in (0, 1/L_f)$ satisfies 
\[
x^*\in \proj_{\cX}(x^*-t\nabla f(x^*)) \quad\quad \forall t\in[0,\T].
\]
This paper may shed a light on developing a gradient-type method  for which any accumulation 
point $x^*$ of the generated sequence satisfies a stronger relation: 
 \[
x^*= \proj_{\cX}(x^*-t\nabla f(x^*)) \quad\quad \forall t\in[0,\T].
\]

\end{document}